\documentclass{article}

\usepackage[table,dvipsnames]{xcolor}

\usepackage{lmodern}

\usepackage[margin=2.5cm]{geometry}
\usepackage[utf8]{inputenc}
\usepackage[T1]{fontenc}
\usepackage{scrbase}
\usepackage{microtype}
\usepackage[shortlabels]{enumitem}
\usepackage{etoolbox}
\usepackage{graphicx}
\usepackage{tabularx}
\usepackage{arydshln}
\usepackage[framemethod=TikZ]{mdframed}
\usepackage{multirow}
\usepackage{makecell}
\usepackage{pdflscape}
\usepackage{authblk}

\usepackage[
backend=biber,
style=alphabetic,
maxnames=99,
backref=true
]{biblatex}
\usepackage{amsmath}
\usepackage{amssymb}
\usepackage{amsthm}
\usepackage{thmtools}
\usepackage{mathtools}
\usepackage{braket}

\usepackage{todonotes}

\usepackage{array}
\newcolumntype{Y}{>{\centering\arraybackslash}X}

\definecolor{betterblue}{RGB}{0, 110, 230}
\definecolor{bettergreen}{RGB}{164, 180, 43}

\usepackage[most]{tcolorbox}
\newtcolorbox{mybox}{
    colback=orange!10, colframe=orange!40, arc=4mm, boxrule=1pt,
    left=6pt, right=6pt, top=6pt, bottom=6pt,
    width=0.95\textwidth, center
}

\usepackage{hyperref}
\hypersetup{%
	pdftoolbar=false,
	pdfmenubar=false,
	colorlinks=true,
	urlcolor={OliveGreen},
	linkcolor={betterblue},
	citecolor={orange}
}
\usepackage[capitalise,noabbrev]{cleveref}

\theoremstyle{plain}
\newtheorem{theorem}{Theorem}[section]
\newtheorem{lemma}[theorem]{Lemma}
\newtheorem{proposition}[theorem]{Proposition}
\newtheorem{corollary}[theorem]{Corollary}

\newtheorem{observation}[theorem]{Observation}
\newtheorem{question}[theorem]{Question}
\newtheorem*{question*}{Question}

\newtheorem*{conjecture*}{Conjecture}

\AtEndEnvironment{proposition}{\setcounter{claim}{0}}
\AtEndEnvironment{lemma}{\setcounter{claim}{0}}
\AtEndEnvironment{theorem}{\setcounter{claim}{0}}

\theoremstyle{definition}
\newtheorem{definition}[theorem]{Definition}
\newtheorem*{notation}{Notation}
\newtheorem{remark}[theorem]{Remark}
\newtheorem{examplex}[theorem]{Example}
\newenvironment{example}
  {\pushQED{\qed}\examplex}
  {\popQED\endexamplex}

\newenvironment{remarkbox}{%
  \mdfsetup{%
    frametitle={%
      \tikz[baseline=(current bounding box.east),outer sep=0pt]
      \node[anchor=east,rectangle,fill=white!20]
    {\strut Remark};}
  }%
  \mdfsetup{%
    innertopmargin=0pt,
    linewidth=.5pt,topline=true,%
    frametitleaboveskip=\dimexpr-\ht\strutbox\relax%
  }
  \begin{mdframed}[]
  }{%
  \end{mdframed}%
}

\renewcommand{\epsilon}{\varepsilon}

\newcommand{\C}{\ensuremath{\mathbb{C}}}

\newcommand{\N}{\ensuremath{\mathbb{N}}}

\DeclareMathOperator{\End}{End}

\DeclareMathOperator{\id}{id}
\DeclareMathOperator{\Tr}{Tr}

\DeclareMathOperator{\lspan}{span}

\DeclareMathOperator{\img}{img}

\DeclareMathOperator{\rk}{rk}
\newcommand\restrict[1]{\vert_{#1}}

\newcommand{\transpose}{^\mathsf{T}}
\newcommand{\adjoint}{^\dagger}
\newcommand{\inverse}{^{-1}}
\newcommand{\identity}{I}
\newcommand{\compl}[1]{#1^c}
\newcommand{\op}{^{\operatorname{op}}}

\newcommand{\IFF}{\Longleftrightarrow}
\renewcommand{\iff}{\Leftrightarrow}

\renewcommand{\implies}{\Rightarrow}
\newcommand{\abs}[1]{\lvert #1 \rvert}
\newcommand{\norm}[1]{\lVert #1 \rVert}
\newcommand{\Cstar}{C\ensuremath{^*}}

\newcommand{\psubseteq}{\sqsubseteq}
\newcommand{\FdHilb}{\mathbf{FdHilb}}
\newcommand{\Hilb}{\mathbf{Hilb}}
\newcommand{\TFdHilb}{\mathbf{2FdHilb}}

\newcommand{\looprm}[1]{#1^{\mathrlap{\mspace{2.3mu}\circ}\times}}
\newcommand{\loops}[1]{#1^{\circ}}

\newcommand{\pproj}[1]{\widetilde{#1}}
\newcommand{\mmap}[1]{\widehat{#1}}
\newcommand{\ideal}[1]{\mathring{#1}}

\newcommand{\ccdot}{\makebox[1ex]{\textbf{$\cdot$}}}

\newcommand{\lrangle}[1]{\langle #1 \rangle}

\newcommand{\ucl}[1]{\operatorname{cl} #1}
\newcommand{\uclarg}[1]{(\operatorname{cl} #1)}

\makeatletter
\let\amsmath@bigm\bigm

\renewcommand{\bigm}[1]{%
  \ifcsname fenced@\string#1\endcsname
    \expandafter\@firstoftwo
  \else
    \expandafter\@secondoftwo
  \fi
  {\expandafter\amsmath@bigm\csname fenced@\string#1\endcsname}%
  {\amsmath@bigm#1}%
}

\newcommand{\DeclareFence}[2]{\@namedef{fenced@\string#1}{#2}}
\makeatother

\DeclareFence{\mid}{|} 

\makeatletter
\DeclareRobustCommand\ket[1]{%
  \@ifnextchar\bra{\k@t{#1}\!}{\k@t{#1}}%
}
\newcommand\k@t[1]{{|{#1}\rangle}}
\makeatother

\makeatletter
\newcommand{\vast}{\bBigg@{3}}
\newcommand{\Vast}{\bBigg@{4}}
\newcommand{\VVast}{\bBigg@{5}}
\newcommand{\VAST}{\bBigg@{6}}
\makeatother

\newcommand{\stoptocwriting}{%
  \addtocontents{toc}{\protect\setcounter{tocdepth}{-5}}}

\makeatletter
\patchcmd{\@maketitle}{\LARGE \@title}{\fontsize{19}{21.2}\selectfont\@title}{}{}
\makeatother

\usetikzlibrary{patterns, patterns.meta}

\input{./style/my_tikzit.sty}

\tikzstyle{myBlue}=[draw={rgb,255: red,25; green,132; blue,232}]
\tikzstyle{myBlueFill}=[fill={rgb,255: red,25; green,132; blue,232}]
\tikzstyle{myGreen}=[draw={rgb,255: red,164; green,180; blue,43}]
\tikzstyle{myPink}=[draw={rgb,255: red,208; green,12; blue,133}]
\tikzstyle{myPinkFill}=[fill={rgb,255: red,208; green,12; blue,133}]
\tikzstyle{myYellow}=[draw={rgb,255: red,255; green,160; blue,0}]
\tikzstyle{myYellowFill}=[fill={rgb,255: red,255; green,160; blue,0}]
\tikzstyle{myBrown}=[draw={rgb,255: red,173; green,100; blue,82}]
\tikzstyle{myBrownFill}=[fill={rgb,255: red,173; green,100; blue,82}]
\tikzstyle{myOrange}=[draw={rgb,255: red,245; green,124; blue,0}]
\tikzstyle{myOrangeFill}=[fill={rgb,255: red,245; green,124; blue,0}]
\tikzstyle{myPurple}=[draw={rgb,255: red,178; green,121; blue,219}]
\tikzstyle{myPurpleFill}=[fill={rgb,255: red,178; green,121; blue,219}]
\tikzstyle{myPlum}=[draw={rgb,255: red,123; green,31; blue,162}]
\tikzstyle{myPlumFill}=[fill={rgb,255: red,123; green,31; blue,162}]
\tikzstyle{myTeal}=[draw={rgb,255: red,0; green,151; blue,167}]
\tikzstyle{map}=[myBlue, fill=white, shape=circle, thick, tikzit category=maps, minimum size=20pt, inner sep=1pt, tikzit draw={rgb,255: red,25; green,132; blue,232}]
\tikzstyle{correlation}=[map, myPink, tikzit draw={rgb,255: red,208; green,12; blue,133}, tikzit category=games]
\tikzstyle{game}=[map, myGreen, tikzit draw={rgb,255: red,104; green,159; blue,56}, tikzit category=games]
\tikzstyle{q-func}=[map, myPink, tikzit draw={rgb,255: red,208; green,12; blue,133}, tikzit category=games]
\tikzstyle{frob}=[fill=white, tikzit draw=black, shape=circle, very thick, tikzit category=frobenius, minimum size=4pt, inner sep=0pt]
\tikzstyle{x-frob}=[frob, myYellow, tikzit draw={rgb,255: red,255; green,160; blue,0}, tikzit category=frobenius]
\tikzstyle{dark-x-frob}=[frob, myYellow, myYellowFill, tikzit fill={rgb,255: red,255; green,160; blue,0}, tikzit draw={rgb,255: red,255; green,160; blue,0}, tikzit category=frobenius]
\tikzstyle{y-frob}=[frob, myBrown, tikzit draw={rgb,255: red,173; green,100; blue,82}, tikzit category=frobenius]
\tikzstyle{xy-frob}=[frob, myOrange, ultra thick, tikzit draw={rgb,255: red,245; green,124; blue,0}, tikzit category=frobenius]
\tikzstyle{half-xy-proj}=[xy-frob, shape=rectangle, minimum size=6pt, tikzit draw={rgb,255: red,245; green,124; blue,0}, tikzit category=frobenius, tikzit fill=white, tikzit shape=rectangle]
\tikzstyle{half-xy-proj-flipped}=[half-xy-proj, myOrangeFill, tikzit draw={rgb,255: red,245; green,124; blue,0}, tikzit category=frobenius, tikzit fill={rgb,255: red,245; green,124; blue,0}, tikzit shape=rectangle]
\tikzstyle{twisted-xy-proj}=[myOrange, left color=white, right color={rgb,255: red,245; green,124; blue,0}, tikzit draw={rgb,255: red,245; green,124; blue,0}, shape=rectangle, thin, tikzit category=frobenius, minimum size=8pt, inner sep=0pt, tikzit fill=white, tikzit shape=rectangle]
\tikzstyle{twisted-xy-proj-flipped}=[myOrange, right color=white, left color={rgb,255: red,245; green,124; blue,0}, tikzit draw={rgb,255: red,245; green,124; blue,0}, shape=rectangle, thin, tikzit category=frobenius, minimum size=8pt, inner sep=0pt, tikzit fill={rgb,255: red,245; green,124; blue,0}, tikzit shape=rectangle]
\tikzstyle{a-frob}=[frob, myPurple, tikzit draw={rgb,255: red,178; green,121; blue,219}, tikzit category=frobenius]
\tikzstyle{dark-a-frob}=[frob, myPurple, myPurpleFill, tikzit fill={rgb,255: red,178; green,121; blue,219}, tikzit draw={rgb,255: red,178; green,121; blue,219}, tikzit category=frobenius]
\tikzstyle{b-frob}=[frob, myBlue, tikzit draw={rgb,255: red,25; green,118; blue,210}, tikzit category=frobenius]
\tikzstyle{ab-frob}=[frob, myPlum, ultra thick, tikzit draw={rgb,255: red,123; green,31; blue,162}, tikzit category=frobenius]
\tikzstyle{half-ab-proj}=[ab-frob, shape=rectangle, minimum size=6pt, tikzit draw={rgb,255: red,123; green,31; blue,162}, tikzit category=frobenius, tikzit shape=rectangle]
\tikzstyle{half-ab-proj-flipped}=[half-ab-proj, myPlumFill, tikzit category=frobenius, tikzit draw={rgb,255: red,123; green,31; blue,162}, tikzit fill={rgb,255: red,123; green,31; blue,162}, tikzit shape=rectangle]
\tikzstyle{twisted-ab-proj}=[myPlum, left color=white, right color={rgb,255: red,123; green,31; blue,162}, tikzit draw={rgb,255: red,123; green,31; blue,162}, shape=rectangle, thin, tikzit category=frobenius, minimum size=8pt, inner sep=0pt, tikzit fill=white, tikzit shape=rectangle]
\tikzstyle{twisted-ab-proj-flipped}=[myPlum, right color=white, left color={rgb,255: red,123; green,31; blue,162}, shape=rectangle, thin, tikzit category=frobenius, minimum size=8pt, inner sep=0pt, tikzit draw={rgb,255: red,123; green,31; blue,162}, tikzit fill={rgb,255: red,123; green,31; blue,162}, tikzit shape=rectangle]
\tikzstyle{trap}=[fill=white, myPink, shape=trapezium, thick, tikzit shape=rectangle, trapezium left angle=90, trapezium right angle=70, minimum height=16pt, minimum width=16pt, inner sep=2pt, draw={rgb,255: red,25; green,132; blue,232}, tikzit category=maps]
\tikzstyle{conj}=[trap, trapezium left angle=70, trapezium right angle=90, tikzit draw={rgb,255: red,25; green,132; blue,232}, tikzit shape=rectangle, tikzit category=maps]
\tikzstyle{dagger}=[conj, shape border rotate=180, tikzit draw={rgb,255: red,25; green,132; blue,232}, tikzit shape=rectangle, tikzit category=maps]
\tikzstyle{transpose}=[trap, shape border rotate=180, tikzit draw={rgb,255: red,25; green,132; blue,232}, tikzit shape=rectangle, tikzit category=maps]
\tikzstyle{bra}=[myTeal, fill=white, tikzit draw={rgb,255: red,0; green,151; blue,167}, regular polygon, regular polygon sides=3, tikzit category=maps, inner sep=0pt, minimum size=20pt]
\tikzstyle{bra_s}=[myTeal, fill=white, tikzit draw={rgb,255: red,0; green,151; blue,167}, regular polygon, regular polygon sides=3, tikzit category=maps, inner sep=-1.25pt, minimum size=20pt]
\tikzstyle{bra_xs}=[myTeal, fill=white, tikzit draw={rgb,255: red,0; green,151; blue,167}, regular polygon, regular polygon sides=3, tikzit category=maps, inner sep=-1.75pt, minimum size=20pt]
\tikzstyle{bra_xxs}=[myTeal, fill=white, tikzit draw={rgb,255: red,0; green,151; blue,167}, regular polygon, regular polygon sides=3, tikzit category=maps, inner sep=-3pt, minimum size=1pt]
\tikzstyle{ket}=[bra, shape border rotate=180, tikzit draw={rgb,255: red,0; green,151; blue,167}, tikzit category=maps]
\tikzstyle{ket_s}=[{bra_s}, shape border rotate=180, tikzit draw={rgb,255: red,0; green,151; blue,167}, tikzit category=maps]
\tikzstyle{ket_xs}=[{bra_xs}, shape border rotate=180, tikzit draw={rgb,255: red,0; green,151; blue,167}, tikzit category=maps]
\tikzstyle{ket_xxs}=[{bra_xxs}, shape border rotate=180, tikzit draw={rgb,255: red,0; green,151; blue,167}, tikzit category=maps]
\tikzstyle{map_multi}=[myGreen, thick, shape=rectangle, tikzit draw={rgb,255: red,25; green,132; blue,232}, tikzit shape=rectangle, minimum width=1.5cm, tikzit category=maps]
\tikzstyle{small_map_multi}=[{map_multi}, minimum width=1cm, tikzit category=maps, shape=rectangle, tikzit draw={rgb,255: red,25; green,132; blue,232}]
\tikzstyle{proj}=[fill=white, shape=rectangle, tikzit category=maps, myGreen]
\tikzstyle{cket}=[bra, shape=cut triangle down, tikzit draw={rgb,255: red,0; green,128; blue,128}, apex angle=110, minimum width=.9cm]
\tikzstyle{cbra}=[cket, tikzit draw={rgb,255: red,0; green,128; blue,128}, shape=cut triangle up]
\tikzstyle{widecket}=[cket, minimum width=1.35cm, apex angle=100]
\tikzstyle{ccket}=[cket, shape=cut triangle down alt, tikzit draw={rgb,255: red,0; green,128; blue,128}, tikzit shape=circle]
\tikzstyle{ccbra}=[cbra, shape=cut triangle up alt, tikzit draw={rgb,255: red,0; green,128; blue,128}, tikzit shape=circle]
\tikzstyle{idendo}=[fill=white, draw={black!80}, shape=circle, dashed, inner sep=0.1]
\tikzstyle{selfadj}=[trap, signal, signal pointer angle=120, signal to=east, fill=white, tikzit shape=rectangle, draw={rgb,255: red,25; green,132; blue,232}]
\tikzstyle{selfadj-conj}=[selfadj, signal to=west, fill=white, tikzit draw=blue, tikzit shape=rectangle]
\tikzstyle{pselfadj}=[selfadj, myGreen, tikzit draw={rgb,255: red,64; green,161; blue,0}, tikzit shape=rectangle]
\tikzstyle{pconjselfadj}=[selfadj, signal to=west, myGreen, tikzit draw={rgb,255: red,64; green,161; blue,0}, tikzit shape=rectangle]
\tikzstyle{pselfadjselfdual}=[tikzit draw={rgb,255: red,64; green,161; blue,0}, signal, signal pointer angle=120, signal to=east and west, fill=white, tikzit shape=circle, myGreen]
\tikzstyle{x-action}=[myOrange, fill=white, shape=rectangle, tikzit draw={rgb,255: red,245; green,124; blue,0}]
\tikzstyle{a-action}=[myPurple, fill=white, shape=rectangle, tikzit draw={rgb,255: red,178; green,121; blue,219}]
\tikzstyle{y-action}=[myBrown, fill=white, shape=rectangle, tikzit draw={rgb,255: red,173; green,100; blue,82}]

\tikzstyle{hilbert}=[-, myTeal, thick, tikzit draw={rgb,255: red,0; green,151; blue,167}]
\tikzstyle{hilbert2}=[-, myGreen, thick, tikzit draw={rgb,255: red,148; green,210; blue,75}]
\tikzstyle{x-wire}=[-, myYellow, very thick, tikzit draw={rgb,255: red,255; green,160; blue,0}]
\tikzstyle{y-wire}=[-, myBrown, very thick, tikzit draw={rgb,255: red,173; green,100; blue,82}]
\tikzstyle{xy-wire}=[-, myOrange, line width=2pt, tikzit draw={rgb,255: red,245; green,124; blue,0}]
\tikzstyle{a-wire}=[-, myPurple, very thick, tikzit draw={rgb,255: red,178; green,121; blue,219}]
\tikzstyle{b-wire}=[-, myBlue, very thick, tikzit draw={rgb,255: red,25; green,118; blue,210}]
\tikzstyle{ab-wire}=[-, myPlum, tikzit draw={rgb,255: red,123; green,31; blue,162}, line width=2pt]
\tikzstyle{hilbertnoarrow}=[-, tikzit draw={rgb,255: red,0; green,151; blue,167}, thick, myTeal]
\tikzstyle{dirwire}=[-]
\tikzstyle{wire}=[-]
\tikzstyle{dirwire02}=[-, tikzit draw=magenta]
\tikzstyle{invis}=[-, draw=none, tikzit draw={rgb,255: red,128; green,128; blue,128}]
\tikzstyle{z-wire}=[-, dashed, draw={rgb,255: red,64; green,64; blue,64}, very thick]
\tikzstyle{dir-a-wire}=[-, a-wire, tikzit draw={rgb,255: red,178; green,121; blue,219}]
\tikzstyle{dir-b-wire}=[-, b-wire, tikzit draw={rgb,255: red,25; green,118; blue,210}]
\tikzstyle{dir-x-wire}=[-, x-wire, tikzit draw={rgb,255: red,255; green,160; blue,0}]
\tikzstyle{dir-y-wire}=[-, y-wire, tikzit draw={rgb,255: red,173; green,100; blue,82}]
\tikzstyle{f-dir-a-wire}=[-, a-wire, tikzit draw={rgb,255: red,178; green,121; blue,219}]
\tikzstyle{f-dir-x-wire}=[-, x-wire, tikzit draw={rgb,255: red,255; green,160; blue,0}]
\tikzstyle{f-dir-y-wire}=[-, y-wire, tikzit draw={rgb,255: red,173; green,100; blue,82}]
\tikzstyle{b-dir-a-wire}=[-, a-wire, tikzit draw={rgb,255: red,178; green,121; blue,219}]
\tikzstyle{b-dir-x-wire}=[-, x-wire, tikzit draw={rgb,255: red,255; green,160; blue,0}]
\tikzstyle{b-dir-y-wire}=[-, y-wire, tikzit draw={rgb,255: red,173; green,100; blue,82}]

\input{./style/tikz_adjustments.sty}

\input{./style/shapes.tex}

\title{Unification of Quantum Graph Properties}

\author[1,2]{Gian Luca Spitzer \thanks{\href{mailto:gian-luca.spitzer@u-bordeaux.fr}{gian-luca.spitzer@u-bordeaux.fr}}}
\affil[1]{LaBRI, Universit\'e de Bordeaux, CNRS, Bordeaux INP, UMR-5800, France}
\affil[2]{Laboratoire de Physique Th\'eorique, Universit\'e de Toulouse, CNRS, UPS, France}

\date{}

\begin{document}

\maketitle

\begin{abstract}
  Many properties of classical graphs are defined in terms of subsets of the vertex set. 
  Examples include connected components, which are subsets $X \subseteq V(G)$ such that $X \times \compl{X}$ and $E(G)$ are disjoint, or independent sets, for which $X \times X$ and $E(G)$ are disjoint.
  Direct generalisations of these definitions to quantum graphs are difficult to achieve, since the natural notion of subsets of a quantum set is much too rigid. As a consequence, approaches to generalising these classical properties to the quantum setting have been eclectic. 
  In some cases, multiple inequivalent definitions of the same notion are in use. We introduce a natural and well-motivated alternative definition of subsets of a quantum set. Building on this, we propose unified definitions of quantum graph properties as straightforward generalisations of the classical definitions. 
  We recover this way the established notions of colourings and connected components.
   For independent sets and cliques, our approach suggests variations that diverge from existing definitions, but address some of their counterintuitive properties. We nevertheless show how to recover an important existing definition of independent sets in our framework. 
\end{abstract}

\tableofcontents

\section{Introduction}
The program of noncommutative discrete mathematics has been quite successful. Rooted in noncommutative geometry, the idea is to show how certain classical mathematical objects are encoded by commutative algebras. One then drops this commutativity requirement and studies the resulting ``noncommutative'' structures. Suprisingly, these structures turn out to have important applications in quantum physics, which, to name but one reason, justify replacing the modifier ``noncommutative'' by ``quantum''. Examples of these structures include (discrete) quantum groups, which inter alia have applications in quantum field theories \cite{chaichian1996introduction}; quantum posets, which are used to define semantics for quantum programming languages \cite{kornell_quantum_2021}; and, central to this paper, \emph{quantum graphs}, which can be used to describe properties of quantum channels \cite{duan_zeroerror_2013}. Many of these structures have moreover been generalised by Kornell under the notion of \emph{discrete quantum structures} \cite{kornell_discrete_2023,kornell_discrete_2024}.

Fundamental to these discrete quantum structures is the notion of a \emph{quantum set}. It is a consequence of Gelfand duality that finite classical sets are in bijective correspondence with commutative finite-dimensional \Cstar-algebras. Dropping the commutativity requirement, we define (finite) quantum sets as arbitrary finite-dimensional \Cstar-algebras.
The classical set on $n$ elements is then recovered as the quantum set $\C^n$. This may also be interpreted as a \emph{linearisation} of the set: It is the complex vector space spanned by its elements, which in turn act as the standard basis.

Classically, a discrete structure is a set equipped with a collection of functions and relations that satisfy certain properties. One approach to defining a quantum analogue of a given structure is to push these properties forward along Gelfand duality and see what properties of the \Cstar-algebra they correspond to. This approach, for example, informs us what functions between quantum sets should look like. In other cases, it is easier to directly linearise a classical definition in a way that it agrees with the original definition on $\C^n$ under the above-mentioned interpretation of set elements as standard basis vectors. Representative of this approach is the definition of binary relations in \cite[Proposition VII.1]{musto_compositional_2018}.

A natural concept to quantise is that of a subset. If one follows the approach of Gelfand duality, one finds that subsets should correspond to \emph{quotients} of the \Cstar-algebra. This definition is widely accepted \cite{kornell_quantum_2020, gromada_examples_2022, kornell_quantum_2026}, but it has a significant problem: The prototypical non-classical quantum set is $M_n$, the algebra of complex $n \times n$ matrices. It is this quantum set that arises in most applications of discrete quantum structures in quantum information theory. That said, $M_n$ is a \emph{simple} algebra. This means that $M_n$ has no non-trivial quotients and hence no non-trivial subsets as a quantum set.

In the context of quantum graphs, this makes the definition essentially useless. A \emph{quantum graph} can be defined as a quantum set of vertices equipped with a binary relation. This is a straightforward quantisation of the classical definition. Naturally, we turn to classical graph theory to inform our study of \emph{quantum} graph theory. In classical graph theory, however, many definitions are phrased in terms of 
subsets:
Connected components, colour classes, independent sets, and cliques are all defined as subsets of the vertex set satisfying certain properties, to name but the most important examples. If one wants to avoid these notions to become trivial for quantum graphs on $M_n$, one has to adopt ad-hoc approaches towards their definition. In particular in the case of independent sets this has led to many non-equivalent definitions being in use. It seems desirable to unify the existing definitions in one framework, and a subset-inspired approach seems the most natural.

Since the approach of quantising subsets through Gelfand duality yields an unsatisfactory definition, a natural step is to instead attempt linearising the notion of subsets directly. We show that proceeding this way, one obtains subsets of a quantum set $X$ as projectors $X \to X$ that are moreover right-module homomorphisms. When viewing the \Cstar-algebra $X$ as a subalgebra of $B(H)$ for some Hilbert space $H$, these projectors are in turn in bijective correspondence with right-modules of $B(H, \C)$ over the commutant of $X$. In this (or similar) phrasing, these \emph{projective subsets} have already shown up in previous work \cite[Appendix B]{kornell_quantum_2020} \cite{kornell_discrete_2023}. However, they have never gotten a standalone treatment, nor have they been used to define properties of quantum graphs.
Kornell in particular calls them \emph{predicates} and explicitly distinguishes them from the notion of subset \cite[p. 350]{kornell_discrete_2023}. We revisit this conception. We show that there are good arguments to really think of these objects as subsets instead of just predicates. We show that they behave a lot like subsets---they admit set-theoretic operations, capture the notion of images and restrictions of classical functions between quantum sets, and enjoy similar equivalent characterisations as classical subsets, for example as unary relations or characteristic vectors. Crucially, they turn out to be the tool we need to unify existing definitions of quantum graph theory.

Overall, our contributions are twofold. Our first contribution is developing the theory of projective subsets. We prove a number of fundamental results that promise to be useful in future work and show that the formalism is nicely compatible with the diagrammatic formalism introduced by Musto, Reutter, and Verdon \cite{musto_compositional_2018} and a bicategorical generalisation that allows refining diagrams, used for example in \cite{verdon_covariant_2022}.

Our second contribution is to use these projective subsets to unify existing definitions of quantum graph properties. Concretely, we give the following definitions in terms of projective subsets.

\begin{itemize}
  \item \textbf{Connectedness:} Existence of a non-trivial $X \psubseteq V(G)$ such that $X \times \compl{X}$ and $E(G)$ are disjoint.
  \item \textbf{Connected Components:} Partitions of $V(G)$ into $X_1, \dots, X_k \psubseteq V(G)$ such that $X_i \times X_j$ and $E(G)$ are disjoint for all $i \neq j \in [k]$.
  \item \textbf{Colourings:} Partitions of $V(G)$ into $X_1, \dots, X_k \psubseteq V(G)$ such that $X_i \times X_i$ and $E(G)$ are disjoint for all $i \in [k]$.
  \item \textbf{Vertex Covers:} $X \psubseteq V(G)$ such that $E(G) \psubseteq (X \times V(G)) \cup (V(G) \times X) \cup (X \times X)$.
\end{itemize}

The first three definitions are equivalent to the established notions in the literature. The fourth is the first definition of vertex covers for quantum graphs to our knowledge. Note that these definitions mirror their classical counterparts: 
If
one replaces the projective subset relation $\psubseteq$ by the ordinary subset relation $\subseteq$, one directly recovers the classical definitions.
In the same spirit, we propose the following definitions for independent sets and cliques, which do not correspond to any existing definition in the quantum graph literature. They satisfy certain desirable properties from the perspective of graph theory: Colourings become partitions of the vertex set into independent sets, a projective subset is an independent set if and only if its complement is a vertex cover, and an independent set of a quantum graph is a clique in the complement graph.

\begin{itemize}
  \item \textbf{Independent Sets:} $X \psubseteq V(G)$ such that $X \times X$ and $E(G)$ are disjoint.
  \item \textbf{Cliques:} $X \psubseteq V(G)$ such that $X \times X \psubseteq E(\loops{G})$.
\end{itemize}

Here, $\loops{G}$ denotes $G$ with ``all loops added''. This action can be phrased in our setting as considering the quantum graph with edge set $E(G) \cup \Delta$, where $\Delta$ is the diagonal relation, cf. \cite[Definition 2.4]{weaver_quantum_2010}.  We moreover show that Weaver's well-studied definition of independent sets \cite{weaver2017quantum} can be recovered in the projective subset formalism. The required definition is of a different flavor than the previous definitions. The previous definitions assume that $G$ is \emph{loopless}, and are extended to arbitrary quantum graphs by considering their loopless forms. On the other hand, to recover Weaver's definition we take into account that $G$ might contain loops. We can also proceed similarly to modify the previously stated clique definition. Perhaps surprisingly, this does not recover Weaver's cliques. Instead it recovers the definition that corresponds to Weaver's independent sets in the complement graph. The original definitions do not satisfy this duality. 
\begin{itemize}
  \item \textbf{Independent Sets (Weaver):} $X \psubseteq V(G)$ such that $(X \times X) \setminus \Delta$ and $E(G)$ are disjoint.
  \item \textbf{Cliques (Dual):} $X \psubseteq V(G)$ such that $(X \times X) \setminus \Delta \psubseteq E(G)$
\end{itemize}

\medskip

The paper is structured as follows. In \cref{sec:prelims}, we recall the necessary background on quantum sets and quantum graphs. We also recall two versions of a graphical calculus: the string diagrams used in the graphical formalism of Musto, Reutter, and Verdon, and their bicategorical generalisation. In \cref{sec:psets}, we introduce projective subsets formally. We present equivalent definitions and their interpretations, highlight the interplay with the diagrammatic formalism, and prove useful fundamental results. In \cref{sec:props}, we then turn our attention to quantum graph properties. We begin by phrasing basic definitions such as graph complements and self-loops in terms of projective subsets. We continue by looking at connected components, colouring, independent sets and cliques, and vertex covers, dedicating a subsection to each. Finally, \cref{sec:conclusion} summarises our findings and discusses potential future work.

\section{Preliminaries}\label{sec:prelims}
First, let us fix some general conventions. For Hilbert spaces $H_1, H_2$, we denote the space of bounded linear functions $H_1 \to H_2$ by $B(H_1, H_2)$ and write $B(H)$ for $B(H, H)$. If $F \in B(H_1, H_2)$ and $H_3 \subseteq H_1$, we write $F(H_3)$ for the subspace $\{F(h) \mid h \in H_3\} \subseteq H_2$. Unless otherwise specified, we assume all Hilbert spaces to be finite-dimensional. We also assume all projectors to be orthogonal projectors: A linear map $P \colon H \to H$ is a \emph{projector} if $P^2 = P\adjoint = P$. Similarly, we call an element $p$ of a \Cstar-algebra a \emph{projection} if $p^2 = p^* = p$. When talking about classical graphs, we assume the vertex set to be $[n] \coloneqq \{1, \dots, n\}$. We also assume all classical graphs to be undirected and loopless. We write $i \sim j$ if the vertices $i$ and $j$ are adjacent.

\subsection{The Graphical Calculus}

Instead of working with the \Cstar-algebras corresponding to quantum sets directly, we will work in an equivalent categorical framework introduced by Vicary \cite{vicary_categorical_2011}. Here, \Cstar-algebras uniquely correspond to special symmetric $\dagger$-Frobenius monoids in $\FdHilb$, the category of finite-dimensional Hilbert spaces. This formulation makes explicit some of the emergent properties of \Cstar-algebras and allows us to reason about quantum sets by means of the powerful string diagrammatic calculus for monoidal categories. In the context of quantum graphs, this formalism was first established by Musto, Reutter, and Verdon \cite{musto_compositional_2018}. The graphical calculus can be refined by means of vertical categorification, passing to the $2$-category $\TFdHilb$. 

\paragraph{The Category FdHilb.}  We begin by defining the simple graphical calculus. $\FdHilb$ is a symmetric monoidal category, whose monoidal product is given by the tensor product and whose braiding morphism is the \emph{swap map} $\sigma$. It is moreover canonically equipped with a contravariant involutive endofunctor $\dagger$ that is the identity on objects, and maps a linear map $f\colon X \to Y$ between Hilbert spaces $X, Y$ to its adjoint $f\adjoint \colon Y \to X$. We can reason formally about any such category in terms of \emph{string diagrams}.
The primitives of the calculus are \emph{strings} (or \emph{wires}) and \emph{boxes}.
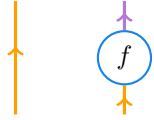
\begin{equation*}
  \begin{tikzpicture}
	\begin{pgfonlayer}{nodelayer}
		\node [style=none] (0) at (0, -0.75) {};
		\node [style=none] (2) at (0, 0.75) {};
	\end{pgfonlayer}
	\begin{pgfonlayer}{edgelayer}
		\draw [style=dir-x-wire, in=270, out=90] (0.center) to (2.center);
	\end{pgfonlayer}
\end{tikzpicture}
 \qquad\quad \begin{tikzpicture}
	\begin{pgfonlayer}{nodelayer}
		\node [style=map] (0) at (0, 0) {$f$};
		\node [style=none] (1) at (0, 0.75) {};
		\node [style=none] (2) at (0, -0.75) {};
	\end{pgfonlayer}
	\begin{pgfonlayer}{edgelayer}
		\draw [style=dir-x-wire] (2.center) to (0);
		\draw [style=dir-a-wire] (0) to (1.center);
	\end{pgfonlayer}
\end{tikzpicture}

\end{equation*}
Strings represent objects, while boxes represent morphisms. In our case the strings represent (elements of) Hilbert spaces, while boxes are linear maps. Composition $g \circ f$ of functions is represented by serial composition of diagrams, while tensor products $f \otimes h$ are represented by parallel composition. Diagrams are read from bottom to top.
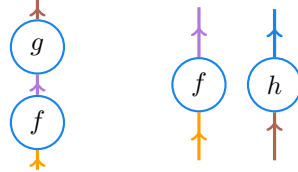
\begin{equation*}
  \input{diagrams/composition.tikz} \qquad\qquad \input{diagrams/tensorprod.tikz}
\end{equation*}
It follows from the representation of the monoidal product that the monoidal unit, $\C$ in the case of $\FdHilb$, has to be represented by the empty diagram. The swap map is represented by crossing two wires. Since $\FdHilb$ is symmetric, it does not matter which wire crosses on top and which on the bottom.
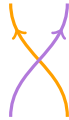
\begin{equation*}
  \input{diagrams/swap.tikz} 
\end{equation*}
Applying the $\dagger$-functor, which in our case correspond to taking adjoints, is represented by mirroring the diagram across a horizontal axis, but keeping the original arrow orientations. For example, we have
\begin{equation*}
  \input{diagrams/composition-dagger.tikz}\enspace.
\end{equation*}
$\FdHilb$ is moreover \emph{rigid}, that is every object $X$ has a dual $X^*$ such that there exist morphisms $\eta\colon \C \to X^* \otimes X$ and $\epsilon \colon X \otimes X^* \to \C$ satisfying the snake equations $(\epsilon \otimes \id)(\id \otimes \eta) = \id_X$ and $(\id \otimes \epsilon)(\eta \otimes \id) = \id_{X^*}$. We represent the dual of an element by the same wire, but with the arrows pointing ``the wrong way''. The duality morphisms can then be represented by 
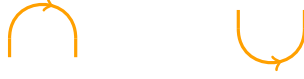
\begin{equation*}
  \input{diagrams/cap.tikz} \qquad\qquad\qquad \input{diagrams/cup.tikz}
\end{equation*}
and the snake equations read
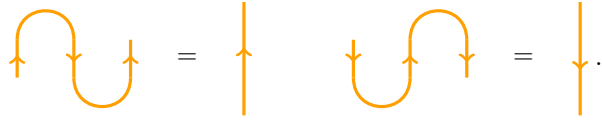
\begin{equation*}
  \input{diagrams/snakeeq-2.tikz} \qquad\qquad \input{diagrams/snakeeq.tikz} \enspace. 
\end{equation*}
Using the duality morphisms we may define the transpose of a linear map.

\begin{definition}\label{def:transpose}
  Let $f\colon X \to Y$ be a linear map. Its \emph{transpose} is the map $f\transpose \colon Y^* \to X^*$ defined as
  \begin{equation}\label{eq:transpose}
    \input{diagrams/transpose.tikz}\enspace.
  \end{equation}
\end{definition}

The power of string diagrams is that they can be used in lieu of symbolic equations to prove equalities. Concretely, every equation that can be derived in the diagrammatic calculus by moving and bending strings, sliding boxes along wires, and rotating boxes also holds in $\FdHilb$, cf. \cite[Theorem 3.28]{heunen_categories_2019}. We say that two diagrams are \emph{isomorphic} if one can be deformed into the other by means of these operations. One such transformation would be to consider the right-hand side of \cref{eq:transpose} and pulling the bent wires taut. This will result in a 180 degree totation of the $f$ box. We conclude that taking the transpose of a linear map correspond to rotating the corresponding diagram by 180 degrees. By using non-symmetric boxes for maps, we may thus get rid of superscripts to represent transposes and adjoints. We let
\begin{equation*}
  \input{diagrams/nonsym-func.tikz}~, \qquad \input{diagrams/nonsym-func-adjoint.tikz}~, \qquad \input{diagrams/nonsym-func-transpose.tikz}~, \qquad \input{diagrams/nonsym-func-conj.tikz}~.
\end{equation*}
The last equation corresponds to the conjugate of a linear map, defined as $\overline{f} = (f\adjoint)\transpose = (f\transpose)\adjoint$. The corresponding operation in the diagrammatic calculus must thus be mirroring across a horizontal axis while keeping arrow orientations the same, followed by a 180 degree rotation (or vice versa). This corresponds to mirroring across a vertical axis while reversing arrow orientations. We can also use the symmetries of boxes to denote invariance under these operations. For example, a self-adjoint map can be represented as follows.
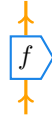
\begin{equation*}
  \begin{tikzpicture}
	\begin{pgfonlayer}{nodelayer}
		\node [style=selfadj] (0) at (0, 0) {$f$};
		\node [style=none] (1) at (0, 0.75) {};
		\node [style=none] (2) at (0, -0.75) {};
	\end{pgfonlayer}
	\begin{pgfonlayer}{edgelayer}
		\draw [style=dir-x-wire] (2.center) to (0);
		\draw [style=dir-x-wire, in=270, out=90] (0) to (1.center);
	\end{pgfonlayer}
\end{tikzpicture}

\end{equation*}
At times this notation introduces too much visual noise, in which case we will revert to circular boxes and superscripts. A special case are morphisms $\C \to X$, which correspond to elements of $X$. Since $\C$ is represented by the empty diagram, these are boxes with one outgoing wire. We represent them as triangles with a corner cut off.
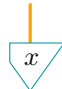
\begin{equation*}
  \begin{tikzpicture}
	\begin{pgfonlayer}{nodelayer}
		\node [style=cket] (0) at (0, -0.25) {$x$};
		\node [style=none] (1) at (0, 0.5) {};
	\end{pgfonlayer}
	\begin{pgfonlayer}{edgelayer}
		\draw [style=x-wire] (0) to (1.center);
	\end{pgfonlayer}
\end{tikzpicture}
 
\end{equation*}
We can now define $\dagger$-Frobenius monoids. 

\begin{definition}
  A \emph{monoid} $X$ in $\FdHilb$ is a finite-dimensional Hilbert space equipped with a linear map $m \colon X \otimes X \to X$, called \emph{multiplication}, and a linear map $u\colon \C \to X$, called \emph{unit}, that are represented by the diagrams
  \begin{equation*}
    m =\enspace \input{diagrams/mult.tikz}  \qquad\qquad\text{\&}\qquad\qquad u =\enspace \begin{tikzpicture}
	\begin{pgfonlayer}{nodelayer}
		\node [style=x-frob] (0) at (0, -0.25) {};
		\node [style=none] (2) at (0, 0.25) {};
	\end{pgfonlayer}
	\begin{pgfonlayer}{edgelayer}
		\draw [style=dir-x-wire] (0) to (2.center);
	\end{pgfonlayer}
\end{tikzpicture}

  \end{equation*}
  and satisfy
  \begin{center}
    \begin{minipage}{0.48\textwidth}
      \centering
      \begin{equation*}
        \input{diagrams/assoc.tikz}
      \end{equation*}
      \small (1) associativity
    \end{minipage}
    \begin{minipage}{0.48\textwidth}
      \centering
      \begin{equation*}
        \input{diagrams/unitality.tikz}\enspace.
      \end{equation*}
      \small (2) unitality
    \end{minipage}
  \end{center}
\end{definition}

As a consequence of the $\dagger$-functor, every monoid in $\FdHilb$ is automatically equipped with a \emph{comultiplication} $m\adjoint \colon X \to X \otimes X$ and a \emph{counit} $u\adjoint \colon X \to \C$. They satisfy the adjoints of conditions (1) and (2), called \emph{coassociativity} and \emph{counitality}. 

\begin{definition}
  A \emph{$\dagger$-Frobenius monoid} is a monoid satisfying the \emph{Frobenius property}
    \begin{equation*}
        \input{diagrams/frob-cond.tikz}\enspace.
    \end{equation*}
  It is called
  \begin{enumerate}[resume]
    \item \emph{special} if 
      \begin{equation*}
        \input{diagrams/special.tikz}\enspace,
      \end{equation*}
    \item \emph{symmetric} if 
      \begin{equation*}
        \input{diagrams/symmetric.tikz}\enspace.
      \end{equation*}
  \end{enumerate}
\end{definition}

Any $\dagger$-Frobenius monoid is self-dual. Indeed, it is a consequence of the Frobenius property and (co-)unitality that the morphisms
\begin{equation*}
  \input{diagrams/monoid-duality-cap.tikz} \qquad\qquad\qquad \input{diagrams/monoid-duality-cup.tikz}
\end{equation*}
satisfy the snake equations. This means that we may omit the arrows on the wires. We are particularly interested in \emph{$\dagger$-SSFMs}: $\dagger$-Frobenius monoids that are both special and symmetric. As the following result shows, these are precisely the categorical objects that correspond to quantum sets.

\begin{proposition}[Theorem 4.7 in \cite{vicary_categorical_2011}]\label{prop:cstarssfmequiv}
  The category of finite-dimensional \Cstar-algebras is equal to the category of special symmetric $\dagger$-Frobenius monoids in $\FdHilb$. In particular, every $\dagger$-SSFM in $\FdHilb$ admits a norm that turns it into a \Cstar-algebra whose involution is given by the conjugate functor.
\end{proposition}

\begin{example}
  As mentioned in the introduction, the classical set with $n$ elements corresponds to the quantum set $\C^n$. This becomes a special symmetric $\dagger$-Frobenius monoid as follows. We choose an orthonormal basis $\ket{1}, \dots, \ket{n}$, the \emph{standard basis}, and define $m$ as the linear extension of
  \begin{equation*}
    \C^n \otimes \C^n \to \C^n, \quad \ket{i} \otimes \ket{j} \mapsto \delta_{ij} \ket{i}.
  \end{equation*}
  In other words, the multiplication on $\C^n$ is componentwise. The unit with respect to this multiplication must then be given by
  \begin{equation*}
    \C \to \C^n, \quad c \mapsto c \mathbf{1} ,
  \end{equation*}
  where
  \begin{equation*}
    \mathbf{1} \coloneqq \sum_{k=1}^n \ket{k}.
  \end{equation*}
  From this, the comultiplication can be determined to be the linear extension of
  \begin{equation*}
    \C^n \to \C^n \otimes \C^n, \quad \ket{i} \mapsto \ket{i} \otimes \ket{i},
  \end{equation*}
  while the counit is the sum-of-entries map $\C^n \to \C$. 
  It is not hard to verify using the above definitions that standard basis elements must be self-conjugate, $\overline{\ket{i}} = (\ket{i}\adjoint)\transpose = \ket{i}$. Representing the self-conjugacy in the diagrammatic calculus, we draw standard basis elements as full triangles instead of triangles with a corner cut off.  
\end{example}

The theory of $\dagger$-SSFMs inherits the following structure theorem from the theory of \Cstar-algebras, which can also be proved independently.

\begin{proposition}[cf. Corollary 5.34 in \cite{heunen_categories_2019}]\label{prop:ssfmstructure}
  Every special symmetric $\dagger$-Frobenius monoid in $X$ in $\FdHilb$ is of the form $X = \bigoplus_{i = 1}^k M_{n_i}$, where for $a, b \in X$ with $a = \bigoplus_i a_i$ and $b = \bigoplus_i b_i$ it holds that
  \begin{align*}
    m(a \otimes b) &= \bigoplus_i \frac{1}{\sqrt{n_i}} a_ib_i\\
    u &= \bigoplus_i \sqrt{n_i} \identity_{n_i} \\
    u\adjoint(a) &= \bigoplus_i \sqrt{n_i} \Tr(a).
  \end{align*}
\end{proposition}

In other words, every $\dagger$-SSFM is a direct sum of matrix algebras, where the multiplication is given by componentwise rescaled matrix multiplication. 

\paragraph{The Category 2FdHilb.} It is often useful to refine the string diagrammatic calculus in order to ``look inside of'' linear maps. Illustrative of this idea is the isomorphism $M_n \cong \C^n \otimes \C^n$ given by the vectorisation map. Graphically, the isomorphism is given by
\begin{equation}\label{eq:mncncniso}
  M_n \to \C^n \otimes \C^n,\qquad \input{diagrams/mn-cncn-iso.tikz}\enspace.
\end{equation}
and allows us to replace a single $M_n$ wire by two parallel $\C^n$ wires. Under this isomorphism, the multiplication and unit maps become
\begin{equation*}
  \input{diagrams/mn-cncn-mult.tikz}  \qquad\qquad\qquad \input{diagrams/mn-cncn-unit.tikz}\enspace.
\end{equation*}
Up to rescaling, this is the \emph{endomorphism monoid} of $\C^n$ in $\FdHilb$, cf. \cite[Definition 3.16]{vicary_categorical_2011}. 
We would like to extend this double-wire approach to arbitrary $\dagger$-SSFMs, but general direct sums of matrix algebras cannot be written as a tensor product of two factors. Instead, we accomplish this by modifying the tensor product, and the correct categorical tool for this are $2$-categories.

For an introduction to $2$-categories, we refer to reader to Chapter 8 in \cite{heunen_categories_2019}. Concretely, we will work with the $2$-category $\TFdHilb$. This category has many equivalent definitions, for the most common one see \cite[Section 8.2]{heunen_categories_2019}. We instead define it as the special case of the category $\mathbf{2Rep}(G)$ introduced in \cite{verdon_covariant_2024} where the group $G$ is trivial. This leads to the following definition.

\begin{definition}
  The $2$-category $\TFdHilb$ is defined by the following data.
  \begin{quote}
    \begin{tabular}{ l l }
      Objects: & $\dagger$-SSFMs in $\FdHilb$,\\
      1-Morphisms: & $\dagger$-bimodules,\\
      2-Morphisms: & bimodule homomorphisms,\\
      Horizontal composition: & interior tensor product, \\
      Vertical composition: & natural composition. \\
    \end{tabular}
  \end{quote}
\end{definition}

For the precise definition of these terms, we refer the reader to \cite{verdon_covariant_2024}. Here, we give an intuitive overview. For $\dagger$-SSFMs $A, B$, a \emph{$\dagger$-$(A, B)$-bimodule} ${_A}X_B$ is an object of $\FdHilb$ that is equipped with a left action of $A$ and a right action of $B$ that are compatible in the obvious way with the respective multiplication, comultiplication, and unit of $A$ and $B$. We also require the actions to commute. Given two $\dagger$-$(A, B)$-bimodules ${_A}X_B$ and ${_A}Y_B$, a \emph{bimodule homomorphism} is a morphism $X \to Y$ that commutes with the actions of $A$ and $B$. The identity $2$-morphisms are precisely given by the corresponding identity maps, which are automatically bimodule homomorphisms. The identity $1$-morphism of an object $A$ is given by ${_A}A_A$, whose left and right actions are given by the $\dagger$-SSFM multiplication.

The crucial definition is that of the \emph{interior tensor product}. Since it is the horizontal composition in our $2$-category, we always take it with respect to an object $B$. To define the interior tensor product ${_A}X \otimes_B Y_C$ of $1$-morphisms ${_A}X_B$ and ${_B}Y_C$, we consider the linear map
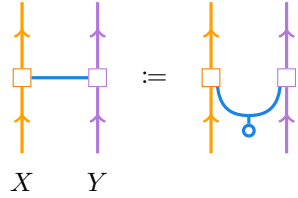
\begin{equation*}
  \input{diagrams/joinmap.tikz}
\end{equation*}
in $\FdHilb$, where the inner wire is the $\dagger$-SSFM $B$ and the rectangular boxes are the left- and right actions of $B$ on $X$ and $Y$ respectively. This map is an idempotent in $\FdHilb$ and thus splits into a unique (up to unitary isomorphism) coisometry $\iota \colon X \otimes Y \to X \otimes_B Y$ and an isometry $\iota\adjoint \colon X \otimes_B Y \to X \otimes Y$ such that
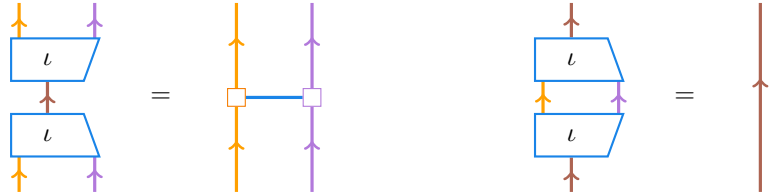
\begin{equation*}
  \input{diagrams/idemp-split-1.tikz} \qquad\qquad\qquad\qquad \input{diagrams/idemp-split-2.tikz}
\end{equation*}
The interior tensor product is defined as the domain of $\iota$. It admits left and right actions as
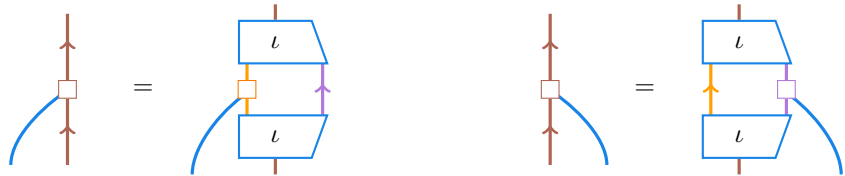
\begin{equation*}
  \input{diagrams/inttprod-left-action.tikz} \qquad\qquad\qquad\qquad \input{diagrams/inttprod-right-action.tikz}
\end{equation*}
and extends to $2$-morphisms as
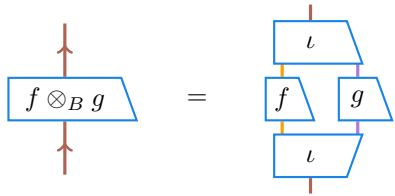
\begin{equation*}
  \input{diagrams/inttprod-2-morphisms.tikz}.
\end{equation*}

As a $2$-category, $\TFdHilb$ admits a diagrammatic calculus, which is defined in terms of \emph{shaded} string diagrams. Its primitives are strings, boxes, and \emph{regions}. Regions represent objects, strings represent $1$-morphisms, and boxes represent $2$-morphisms. Wires that represent identity $1$-morphisms, and boxes that represent identity $2$-morphisms are invisible.
The idea is that a single string partitions the diagram into a left and a right region. If we label the regions by objects $A$ and $B$, then the string must be a $1$-morphism $A \to B$. 
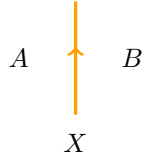
\begin{equation*}
  \input{diagrams/shd-1-morphism.tikz}
\end{equation*}
Instead of labelling regions with letters, we will usually shade them with different patterns. Since identity $1$-morphisms are invisible, we denote their endomorphisms by floating boxes with a dashed outline.
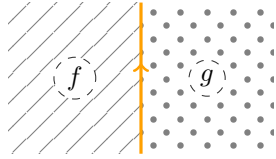
\begin{equation*}
  \input{diagrams/shd-shading.tikz}
\end{equation*}
Vertical composition of $2$-morphisms is given by serial composition of diagrams, mirroring the case of simple string diagrams. Parallel composition of diagrams is only defined if the rightmost region of the left-hand diagram is equal to the leftmost region of the right-hand diagram. In that case it represents the interior tensor product along the corresponding object.
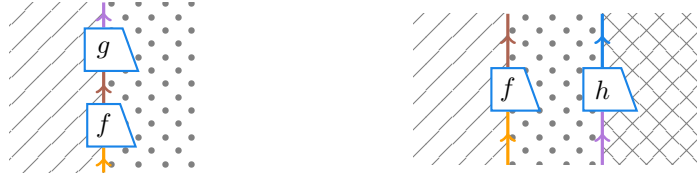
\begin{equation*}
  \input{diagrams/shd-vertical-comp.tikz} \qquad\qquad\qquad\qquad \input{diagrams/shd-horiz-comp.tikz}
\end{equation*}
Since $1$-morphisms in $\TFdHilb$ are still objects, and $2$-morphisms are morphisms in $\FdHilb$, the string-diagrammatic part of the shaded calculus works exactly the same as for simple string diagrams. In particular, $2$-morphisms admit adjoints, conjugates, and transposes, which are again respectively represented by mirroring across vertical and horizontal axes, and rotations by 180 degrees. Since $\FdHilb$ is rigid, $1$-morphisms have duals: If $X$ is a $1$-morphism $A \to B$, then $X^*$ is a $1$-morphism $B \to A$ with left and right actions given by
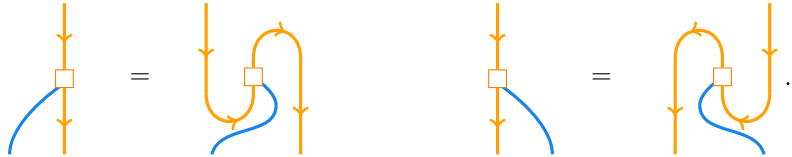
\begin{equation*}
  \input{diagrams/dual-left-action.tikz} \qquad\qquad\qquad\quad \input{diagrams/dual-right-action.tikz}\enspace.
\end{equation*}
Just as for the simpler case of string diagrams, one can prove that every equation that can be derived in the shaded diagrammatic calculus also holds in $\TFdHilb$, cf. \cite[Theorem 8.7]{heunen_categories_2019}.

To introduce the generalised double-wire framework, we note that $\TFdHilb$ has a distinguished object: the trivial $\dagger$-SSFM $\C$. We will represent $\C$ as an unshaded region. It is not hard to see that we must have $\End(\C) \cong \FdHilb$. Indeed, every Hilbert space is automatically a $\dagger$-$(\C, \C)$-bimodule. Moreover, if $X$ is a $1$-morphism $\C \to A$, then $X \otimes_A X^*$ is an object of $\End(\C) \cong \FdHilb$ which can be equipped with the following \emph{multiplication} and \emph{unit} morphisms.
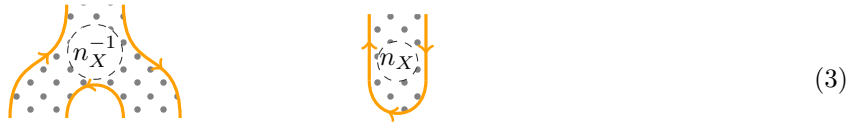
\begin{equation}\label{eq:pairofpants}
  \input{diagrams/2-mult.tikz} \qquad\qquad\qquad\quad \input{diagrams/2-unit.tikz}
\end{equation}
The endomorphism $n_X$ is the square root of the so called left-dimension $N_X$ of $X$, which is given in terms of the duality morphisms of $X$ as
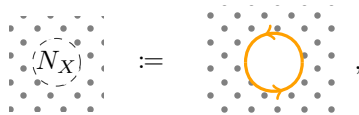
\begin{equation*}
  \input{diagrams/shd-leftdim.tikz}\enspace,
\end{equation*}
see \cite[Section 2]{verdon_covariant_2024} for more details. Crucially, these morphisms give $X \otimes X^*$ the structure of a $\dagger$-SSFM in $\FdHilb$, and up to isomorphism, every $\dagger$-SSFM is of this form, cf. \cite[Corollary 3.37]{chen2022q}.

There is a canonical choice for $X$ and $A$. Let $Y$ be an arbitrary $\dagger$-SSFM in $\FdHilb$. By \cref{prop:ssfmstructure}, we have $Y = \bigoplus_{i = 1}^k M_{n_i}$. We let $X = \bigoplus_{i = 1}^k \C^{n_i}$ and $A = \C^k$. We equip $X$ with the structure of a $\dagger$-$(\C, A)$-module by letting $A$ act by summand-wise scalar multiplication. It follows that $A$ acts the same way on $X^* = \bigoplus_{i = 1}^k (\C^{n_i})^*$. Now consider again the map
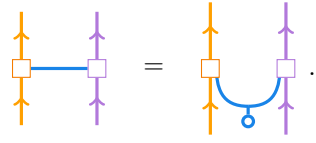
\begin{equation*}
  \input{diagrams/joinmap-simpl.tikz}\enspace.
\end{equation*}
The unit and comultiplication is that of $\C^k$, so comultiplying the unit yields the element $\sum_{i = 1}^k \ket{i} \otimes \ket{i}$. Acting on $X$ and $X^*$ thus precisely annihilates the diagonal elements, which yields $\bigoplus_{i, j = 1}^k \delta_{ij}~ \C^{n_i} \otimes (\C^{n_i})^*$. Consequently, splitting this map yields the coisometry $\iota \colon X \otimes X^* \to \bigoplus_{i = 1} \C^{n_i} \otimes (\C^{n_i})^* = \bigoplus_{i = 1} M_{n_i}$ which drops the off-diagonal summands. It follows that $X \otimes_A X^* = Y$ as desired. Graphically, we recover
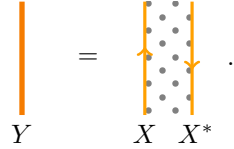
\begin{equation*}
  \input{diagrams/shd-doublewire.tikz}\enspace.
\end{equation*}
The endomorphism $n_X \in \End(\id_A) \cong B(\C^k)$ can be determined to act as multiplication by the vector $[\sqrt{n_1}, \dots, \sqrt{n_k}] \in \C^k$, which recovers precisely the normalisation of \cref{prop:ssfmstructure}. 
There are two special cases of this construction. The first case is when $Y = M_n$. In this case, we have $A = \C$ and the above diagram becomes
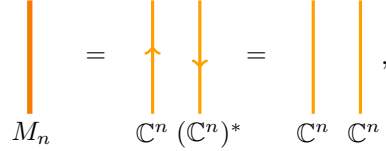
\begin{equation*}
  \input{diagrams/mn-doublewire.tikz}, 
\end{equation*}
recovering the endomorphism monoid of $\C^n$ in $\FdHilb$. Since we know that $\C^n$ is self-dual, we may or may not annotate the wires with arrows. The other case is when $Y = \C^n = \bigoplus_{i = 1}^n \C$. In this case, we have $X = A = X \otimes_A X^* = \C^n \cong (\C^n)^*$, and graphically
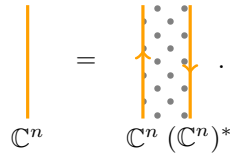
\begin{equation*}
  \input{diagrams/cn-doublewire.tikz}\hspace{-.2cm}.
\end{equation*}
It follows that we may simply collapse or ``glue together'' $1$-morphisms $\C^n$ and $(\C^n)^*$ along the object $\C^n$. We will include this procedure as a valid transformation of shaded diagrams.

Finally,  the elements of $\bigoplus_{i = 1}^k M_{n_i}$ are naturally viewed as endomorphisms of $\bigoplus_{i = 1}^k \C^{n_i}$, in which case they are automatically $(\C, \C^k)$-bimodule homomorphisms. The isomorphism given in \cref{eq:mncncniso} thus generalises as
\begin{equation*}
  \bigoplus_{i = 1}^k M_{n_i} \to \bigoplus_{i = 1}^k \C^{n_i} \otimes_{\C^k} \bigoplus_{i = 1}^k (\C^{n_i})^*,\qquad \input{diagrams/shd-cncn-iso.tikz}\enspace.
\end{equation*}

\subsection{Quantum Sets}

With the categorical background out of the way, we can now turn our attention to quantum sets. Recall that, as motivated through Gelfand duality, quantum sets are precisely finite-dimensional \Cstar-algebras. By \cref{prop:cstarssfmequiv}, we may use the following equivalent definition.

\begin{definition}
  A \emph{quantum set} is a special symmetric $\dagger$-Frobenius monoid in $\FdHilb$.
\end{definition}

It also follows from Gelfand duality that every function $X \to Y$ between classical sets gives rise to a $*$-homomorphism $C(Y) \to C(X)$ of the corresponding commutative \Cstar-algebras. By passing to the categorical setting, we find that these functions correspond to morphisms in $\FdHilb$ that preserve the $\dagger$-SSFM structure. Applying the $\dagger$-functor then recovers a morphism that points in the correct direction. This yields the following definition of functions between quantum sets. 

\begin{definition}\label{def:cfunc}
  Let $X, Y$ be quantum sets. A \emph{classical function} $f\colon X \to Y$ is a linear map $X \to Y$ satisfying
  \begin{equation*}
    \input{diagrams/cohom-comult.tikz} \qquad\quad\qquad \input{diagrams/cohom-counit.tikz} \qquad\quad\qquad \input{diagrams/hom-real.tikz}\enspace.
  \end{equation*}
\end{definition}

We call them ``classical functions'' to dinstinguish them both from the ordinary notion of functions, for example between the underlying Hilbert spaces; and from the much more general notion of \emph{quantum functions} \cite[Definition III.11]{musto_compositional_2018}, which are quantisations of classical functions between quantum sets.  

We continue by defining the tensor product of quantum sets.

\begin{definition}
  The \emph{tensor product} $X \otimes Y$ of quantum sets $X$, $Y$ is the $\dagger$-SSFM given by the tensor product of the underlying Hilbert spaces, with multiplication and unit defined as
  \begin{equation*}
    \input{diagrams/tprod-mult.tikz} \qquad\qquad\qquad \input{diagrams/tprod-unit.tikz}\enspace.
  \end{equation*}
\end{definition}

Finally, we state the established definition of quantum subsets, which will later be the starting point of our investigations.

\begin{definition}[cf. Definition 2.7 in \cite{gromada_examples_2022}]
  Let $X$ be a quantum set. A quantum set $Y$ is a \emph{subset} of $X$, in symbols $Y \subseteq X$, if $Y$ is a \Cstar-algebra quotient of $X$.
\end{definition}

This definition is quite rigid. Recall \cref{prop:ssfmstructure} that every quantum set is decomposes as a direct sum of matrix algebras. It is well-known that matrix algebras are simple, that is they do not have any non-trivial quotients. This implies that the subsets of a quantum set $\bigoplus_{i = 1}^k M_{n_i}$ are given precisely by the quantum sets $\bigoplus_{i \in S} M_{n_i}$, where $S \subseteq [k]$. In particular, the quantum set $M_n$ has no non-trivial subsets.

\subsection{Quantum Graphs}

We can now define quantum graphs. There are multiple equivalent definitions of classical graphs that we may quantise. One of them is to define a graph $G$ as a set $V(G)$ of vertices together with an \emph{adjacency matrix} $A_G \in \{0, 1\}^{V(G) \times V(G)}$. This leads to the following quantisation.

\begin{definition}[cf. Definition V.1 in \cite{musto_compositional_2018}]\label{def:qgraphsadj}
  A \emph{quantum graph} $G$ is a tuple $(V(G), A_G)$, where $V(G)$ is a quantum set and $A_G\colon V(G) \to V(G)$ is a linear map, called the \emph{adjacency operator}, satisfying
  \begin{equation*}
    \input{diagrams/adj-op-schur-idemp.tikz} \qquad\quad\quad\qquad \input{diagrams/adj-op-real.tikz}
  \end{equation*}
\end{definition}

The first condition captures that $A_G$ only has entries in $\{0, 1\}$. If the wires are $\C^n$, it says precisely that $A_G$ is idempotent under the entrywise (Schur) product. The second condition says that $A_G$ is real. In the classical case, it says precisely that $A_G$ has real entries in the standard basis.\footnote{The second condition is of course redundant in the classical case, but this is no longer true for general quantum sets.} Similarly, we may define the notions of being undirected and having loops. The former is equivalent to $A_G$ being symmetric, while the latter can be phrased in terms of the Schur product with the identity.

\begin{definition}
  A quantum graph $G$
  \begin{enumerate}
    \item is \emph{undirected} if
      \begin{equation*}
        \input{diagrams/adj-op-undirected.tikz}
      \end{equation*}
    \item has \emph{no loops} if
      \begin{equation*}
        \input{diagrams/adj-op-irreflexive.tikz}\hspace{.15cm}
      \end{equation*}
    \item has \emph{loops at every vertex} if
      \begin{equation*}
        \input{diagrams/adj-op-reflexive.tikz}~.
      \end{equation*}
  \end{enumerate}
\end{definition}

Instead of defining a classical graph through its adjacency matrix, we may also define it in terms of an \emph{edge relation}, which is a binary relation on the vertex set. There are multiple equivalent quantisations of binary relations, we use the categorical approach of \cite{musto_compositional_2018}. This leads to the following alternative definition of quantum graphs.

\setcounter{theorem}{\value{theorem}-2}
\let\oldthedefinition\thedefinition
\renewcommand{\thedefinition}{\oldthedefinition'}
\begin{definition}
  A quantum graph $G$ is a tuple $(V(G), E(G))$, where $V(G)$ is a quantum set and $E(G)\colon V(G) \otimes V(G) \to V(G) \otimes V(G)$ is a projector, called the \emph{edge relation}, satisfying
  \begin{equation*}
    \input{diagrams/edgerel-bimod-cond.tikz}
  \end{equation*}
\end{definition}
\setcounter{theorem}{\value{theorem}+1}
\let\thedefinition\oldthedefinition

The condition on the projector guarantees that for classical quantum sets, $E(G)$ projects precisely onto a space spanned by a subset of $\{\ket{i} \otimes \ket{j} \mid i, j \in [n]\}$. This is the intuitive linearisation of a binary relation viewed as a subset of $V(G) \times V(G)$. We will revisit this in more detail in \cref{ssec:anyothername}. The notions of undirectedness and having loops take the following form.

\setcounter{theorem}{\value{theorem}-1}
\let\oldthedefinition\thedefinition
\renewcommand{\thedefinition}{\oldthedefinition'}
\begin{definition}
   A quantum graph $G$
   \begin{enumerate}
     \item is undirected if 
       \begin{equation*}
         \input{diagrams/edgerel-symmetric.tikz}\hspace{1.4cm}
       \end{equation*}
     \item has no loops if
       \begin{equation*}
         \input{diagrams/edgerel-irreflexive.tikz}\hspace{.6cm}
       \end{equation*}
     \item has loops at every vertex if
       \begin{equation*}
         \input{diagrams/edgerel-reflexive.tikz}
       \end{equation*}
   \end{enumerate}
\end{definition}  
\let\thedefinition\oldthedefinition

Crucially, the two definitions of quantum graphs are fully equivalent, and there exists a graphical translation between the two points of view. 

\begin{proposition}[Theorem VII.7 in \cite{musto_compositional_2018}]\label{prop:adjopedgerelequiv}
  There is a bijective correspondence between adjacency operators $A_G \colon V(G) \to V(G)$ and edge relations $E(G)\colon V(G) \otimes V(G) \to V(G) \otimes V(G)$ given by
  \begin{equation*}
    \input{diagrams/adjop-to-edgerel.tikz} \qquad\qquad\qquad\quad \input{diagrams/edgerel-to-adjop.tikz}
  \end{equation*}
  such that the notions of undirectedness, having no loops, and having loops at every vertex coincide.
\end{proposition}

We mirror the classical graph-theoretic treatment and always assume that our quantum graphs are undirected and loopless unless otherwise specified.

\section{Projective Subsets}\label{sec:psets}
The notion of a quantum subset is quite restrictive. We have seen that it corresponds to the action of decomposing the quantum set as a direct sum of matrix algebras and then taking a subset of direct summands. In particular, this means that quantum sets that are matrix algebras themselves have no non-trivial subsets. On the other hand, for classical sets $\C^n \cong \C^{\oplus n}$ the direct summands correspond to the standard basis vectors and thus to the set elements, so selecting a subset of them coincides with the classical definition of a subset. In some sense, the direct summands are the ``classical elements'' of a quantum set and taking subsets this way is a very ``classical'' operation. 

We introduce a separate, finer variant of subsets that is non-trivial for quantum sets that are full matrix algebras. The idea is to consider a subset of a classical set as the subspace spanned by a subset of the standard basis vectors of $\C^n$. Such a subspace can be encoded as the orthogonal projector onto it. The condition that the subspace be spanned by a subset of the standard basis is equivalent to the projector being a \emph{module homomorphism} when viewing $\C^n$ as a module over itself. Indeed, if a projector $P\colon \C^n \to \C^n$ is a module homomorphism, we have
\begin{equation*}
  P(e_i)e_j = P(e_ie_j) = \delta_{ij} P(e_i),
\end{equation*}
which is only satisfied if $P(e_i)$ is either $0$ or $e_i$. In other words, every standard basis vector is either contained in the image of $P$ or orthogonal to it. It follows that the image of $P$ must be spanned by the standard basis vectors that do not get annihilated. 
Over a commutative algebra like $\C^n$ left and right modules coincide, but in order to generalise this concept to arbitrary quantum sets, we fix a side by convention.

\begin{definition}\label{def:modcond}
  Let $X$ be a quantum set. A linear map $f\colon X \to X$ satisfies the \emph{module condition} if it is a module homomorphism of $X$ viewed as a right-module over itself, or graphically
  \begin{equation*}
    \input{diagrams/module-condition.tikz}\enspace.
  \end{equation*}
\end{definition}

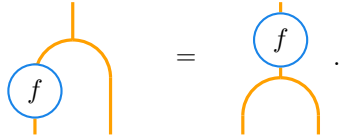
\begin{definition}[Projective Subsets]
  Let $X$ be a quantum set. A \emph{projective subset} $Y \psubseteq X$ of $Y$ is given by a projector $X \to X$ satisfying the module condition. 
\end{definition}

Now there is a well-known isomorphism $\End_R(R) \cong R$ of right-module homomorphisms over a ring $R$ and $R$ itself. In the case of $\dagger$-SSFMs, this isomorphism furthermore restricts to a one-to-one correspondence between projectors on- and projections in the $\dagger$-SSFM.

\begin{lemma}\label{lem:porpionbij}
  Let $X$ be a quantum set. There is a one-to-one correspondence between projectors $P\colon X \to X$ satisfying the module condition and projections $p \in X$.
  \begin{proof}
    We let 
    \begin{equation*}
      \input{diagrams/x-right-mod-x-iso-1.tikz} \qquad\qquad\qquad \input{diagrams/x-right-mod-x-iso-2.tikz}\enspace.
    \end{equation*}
    Then the two definitions are mutually inverse. We have
    \begin{equation*}
      \input{diagrams/por-pion-por.tikz}\enspace,
    \end{equation*}
    while the other direction follows immediately from the unit property. If $P$ is an idempotent with respect to composition, then $p$ is an idempotent with respect to the multiplication in $X$.
    \begin{equation*}
      \input{diagrams/por-idemp-to-pion-idemp.tikz}
    \end{equation*}
    If $P$ is self-adjoint, then $p$ satisfies $p^* = p$.
    \begin{equation*}
      \input{diagrams/por-selfadj-to-pion-selfadj.tikz}
    \end{equation*}
    Conversely, if $p$ is an idempotent then so is $P$.
    \begin{equation*}
      \input{diagrams/pion-idemp-to-por-idemp.tikz}
    \end{equation*}
    Finally, if $p$ satisfies $p^* = p$ then $P$ is self-adjoint.
    \begin{equation*}
      \input{diagrams/pion-selfadj-to-por-selfadj.tikz}\qedhere
    \end{equation*}
  \end{proof}
\end{lemma}

It follows that we may equivalently think of a projective subset $Z \psubseteq X$ as a projection $z \in X$. This duality has a nice interpretation: If $X$ is a classical set and $Z \subseteq X$, then the projector corresponding to $Z$ projects onto $\lspan \{e_i \mid i \in Z\}$, while the projection $z \in X$ corresponding to $Z$ is given by
\begin{equation*}
  z = Z \sum_{x \in X} e_i = \sum_{i \in Z} e_i.
\end{equation*}
It follows that $z$ is precisely the \emph{characteristic vector} of $Z$. 

\begin{notation}
  We will distinguish between the projective subsets themselves, which we will denote by capital letters $X, Y, Z$, their underlying module homomorphisms $\mmap{X}, \mmap{Y}, \mmap{Z}$, and their corresponding projections $\pproj{X}, \pproj{Y}, \pproj{Z}$. A useful mnemonic to distinguish between the two notations is that the hat $\widehat{\phantom{x}}$ is reminiscent of the graphical notation of the multiplication map, and every module homomorphism is given by multiplication with some element.
\end{notation}

\subsection{Projective Subsets by Any Other Name}\label{ssec:anyothername}

We have seen that one may view the projection representation of a projective subset as a generalisation of its characteristic vector. A natural question is whether the corresponding projector also admits such an interpretation. It turns out that it does. In \cite{musto_compositional_2018}, Musto, Reutter, and Verdon define \emph{binary relations} on quantum sets.

\begin{definition}[Definition VII.2 in \cite{musto_compositional_2018}]\label{def:qrel}
  A \emph{binary relation} between two quantum sets $X$ and $Y$ is given by a projector $P\colon X \otimes Y \to X \otimes Y$, satisfying the \emph{bimodule condition}
  \begin{equation*}
    \input{diagrams/bimod-condition.tikz} \quad.
  \end{equation*}
  A \emph{binary relation on a quantum set} $X$ is a relation between $X$ and itself.
\end{definition}

Observe that the bimodule condition is simply the module condition on $X\op \otimes Y$, and hence a relation between $X$ and $Y$ is simply a projective subset of $X\op \otimes Y$. Following this line of thinking, we may view the projector corresponding to a projective subset $Z \psubseteq X$ as a \emph{unary relation} on $X$. This precisely mirrors the classical case, where binary relations are usually thought of as subsets of $X \times Y$ and unary relations as subsets of $X$.

\begin{example}
  An important example for this perspective is the edge relation of a quantum graph. Recall that for a quantum graph $G$, the edge relation $E(G)$ is precisely a binary relation on $V(G)$. Our observations now imply that the edge relation may equivalently be seen as a projective subset of $V(G)\op \otimes V(G)$. 
\end{example}

A very similar perspective is to view a subset of a classical set $X$ as a binary relation between a singleton set and $X$. This perspective also survives the generalisation to the quantum case: The classical singleton set is given by $\C$, so a binary relation between $\C$ and $X$ is a projector $P\colon \C \otimes X \to \C \otimes X$ satisfying the bimodule condition. The bimodule condition on the $\C$ part is trivial, and $\C$ is the monoidal unit in $\FdHilb$, that is $\C \otimes X \cong X$. $P$ is thus equivalently a projector $P\colon X \to X$ satisfying the module condition, and hence gives a projective subset of $X$. This is also nicely seen in the graphical calculus: Wires corresponding to the monoidal unit $\C$ are invisible, so the diagram in \cref{def:qrel} immediately reduces to that in \cref{def:modcond}.

Now there is actually a second established definition of binary quantum relations due to Weaver \cite{weaver_quantum_2010} that precedes Musto, Reutter and Verdon's definition by almost a decade. This definition is of a much more operator-algebraic flavor, but turns out to be equivalent \cite[Proposition VII.5]{musto_compositional_2018}. We can also nicely relate our projective subsets to Weaver's definition of binary relations. Originally, Weaver's definition is stated in terms of von Neumann algebras, but since we are only interested in the finite-dimensional case, the definition simplifies to the following statement about \Cstar-algebras.

\begin{definition}[cf. Definition 2.1 in \cite{weaver_quantum_2010}]\label{def:weaverqrel}
  Let $X$ be a quantum set, viewed as a \Cstar-algebra $X \subseteq B(H)$ for a finite-dimensional Hilbert space $H$. A \emph{binary quantum relation in the sense of Weaver} is a bimodule over its commutant, that is a subspace $S \subseteq B(H)$ satisfying $X'SX' \subseteq S$.
\end{definition}

A priori, this definition seems to depend on the choice of embedding of $X$ into $B(H)$. However, Weaver shows that for any two choices of embedding, there is a one-to-one correspondence between the resulting quantum relations \cite[Theorem 2.7]{weaver_quantum_2010}. We can express the connection between \cref{def:qrel} and \cref{def:weaverqrel} nicely in the shaded graphical calculus of $\TFdHilb$. Since $X$ is a finite-dimensional \Cstar-algebra, it is isomorphic to a direct sum of matrix algebras $X = \bigoplus_{i = 1}^k M_{n_i}$, and hence its commutant is isomorphic to $\C^k$. 
Without loss of generality, we may assume that $H = \C^N$, where $N = \sum_{i = 1}^k n_i$. 
Now $B(H)$ itself is clearly a bimodule over $X'$, so we may express it graphically as the following $1$-morphism $\C^k \to \C^k$ in $\TFdHilb$, where the dotted areas represent $X'$.
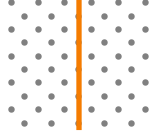
\begin{equation*}
  \input{diagrams/x-bimod-wire.tikz}  
\end{equation*}
Subspaces of $B(H)$ can be encoded by projectors, and if the subspace is a module over a suitable algebra then the projector is a module homomorphism, as the following lemma shows.

\begin{lemma}\label{lem:modproj}
  Let $Y \subseteq B(H)$ be a $*$-subalgebra and $X \subseteq B(H)$ a subspace. Then the following are equivalent.
  \begin{enumerate}
    \item $X$ is a (right, left, bi-) module over $Y$.
    \item The projector onto $X$ is a (right, left, bi-) module homomorphism over $Y$.
  \end{enumerate}
  \begin{proof}
    Without loss of generality, we show the statement for right modules. The other cases are completely analogous. The implication $2 \implies 1$ is straightforward. Let $P\colon B(H) \to B(H)$ be the projector onto $X$ and let $x \in X$. Then for any $y \in Y$ we have
    \begin{equation*}
      xy = (Px)y = P(xy) \in \img P = X.
    \end{equation*}
    For the other direction, note that $B(H)$ admits a canonical inner product as $\lrangle{x, y} = \Tr x^*y$. It follows that the notion of orthogonal complements of subspaces is well-defined. We claim that if $X$ is a right module over $Y$ then so is its orthogonal complement $X^\bot$. Indeed, suppose $x' \in X^\bot$. Then for all $y \in Y$ and $x \in X$ we have
    \begin{equation*}
      \lrangle{x, x'y} = \Tr x^*x'y = \Tr yx^*x' = \Tr\, (xy^*)^*x' = \lrangle{xy^*, x'} = 0,
    \end{equation*}
    where the last equality follows from the fact that $y^* \in Y$ and hence $xy^* \in X$ which is orthogonal to $x'$. It follows that $x'y$ is orthogonal to all $x \in X$ thus $x'y \in X^\bot$, making $X^\bot$ a right module over $Y$. We now show the implication $1 \implies 2$. Let $a \in B(H)$ and $y \in Y$ be arbitrary. We have $a = x + x'$ where $x \in X$ and $x \in X^\bot$. Since $P$ annihilates elements in $X^\bot$, we get
    \begin{equation*}
      P(ay) = P(xy + x'y) = xy = P(a)y,
    \end{equation*}
    where the second equality follows from the fact that both $Y$ and $Y^\bot$ are right modules over $Z$. This concludes the proof.
  \end{proof}
\end{lemma}

Consequently, a bimodule $S \subseteq B(H)$ over the commutant $X'$ is encoded by a projector $P_S \colon B(H) \to B(H)$ that is an $X'$-bimodule homomorphism. This makes it precisely a $2$-morphism $B(H) \to B(H)$ in $\TFdHilb$. A quantum relation according to Weaver is thus given graphically by the following $2$-morphism.
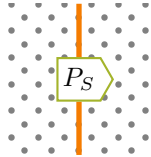
\begin{equation*}
  \input{diagrams/x-bimod-proj.tikz} 
\end{equation*}

Let us now show how one recovers Musto, Reutter, and Verdon's binary relations as projectors $X \otimes X \to X \otimes X$ satisfying the bimodule condition. Since $H$ is finite-dimensional, we have $B(H) \cong H^* \otimes H$. $H$ admits the structure of a right module over $X'$ by interpreting $\C^N = \bigoplus_{i = 1}^k \C^{n_i}$ and letting $X' \cong \C^k$ act by scalar multiplication on the $k$ summands. This allows us to refine the diagram somewhat, replacing the thick $B(H)$ wire by two $H$-wires.
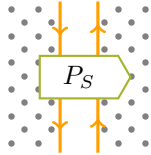
\begin{equation*}
  \input{diagrams/x-bimod-doublewire.tikz} 
\end{equation*}
We now define a projector $P \colon X \otimes X \to X \otimes X$ by horizontal composition with $H$ from the left and with $H^*$ from the right,
\begin{equation*}
  \input{diagrams/mrv-bimod-doublewire.tikz}  
\end{equation*}
Recall that the composition $H \otimes_{\C^k} H^*$ of the $1$-morphisms $H\colon \C \to \C^k$ and $H^* \colon \C^k \to \C$ recovers precisely $X$ as an element of $\End(\mathbf{1}) \cong \FdHilb$, so $P$ is indeed a projector $X \otimes X \to X \otimes X$. Moreover, the pair of pants maps defined in \cref{eq:pairofpants} recover the $\dagger$-SSFM structure of $X$, and it is easy to see this way that $P$ satisfies the bimodule condition.
\begin{equation*}
  \input{diagrams/mrv-bimod-condition.tikz} 
\end{equation*}
Conversely, every map $P \colon X \otimes X \to X \otimes X$ satisfying the bimodule condition also satisfies
\begin{equation*}
  \input{diagrams/mrv-bimod-trace.tikz}
\end{equation*}
as can be seen by plugging in suitable duality morphisms for $H$, connecting the second and third wires, as well as the sixth and seventh wires at the bottom. We thus recover a projector onto a bimodule over the commutant by defining 
\begin{equation*}
  \input{diagrams/mrv-bimod-correspond.tikz}\enspace,
\end{equation*}
and it is easy to see that the two constructions are mutually inverse. We remark that in the case of $X = B(H) = M_N$, our construction reduces precisely to that employed in Proposition VII.11 in \cite{musto_compositional_2018}. There, the authors show that \emph{operator spaces} on $H$---subspaces of $B(H)$---are in bijective correspondence with quantum relations on $B(H)$. In our case, the bicategorical structure allowed us to talk about arbitrary \Cstar-algebras $X \subseteq B(H)$ and quantum relations in the sense of Weaver instead of operator spaces. In this sense, we also show that operator spaces are precisely quantum relations in the sense of Weaver when $X = B(H)$, in which case the commutant is isomorphic to $\C$ and the shading in the diagrams vanishes.

Let us now turn our attention back to projective subsets. Along similar lines as the argument above, we can prove that our definition of projective subsets is equivalent to the following alternative definition in the spirit of Weaver's quantum relations.

\begin{definition}\label{def:projsubsetweaver}
  Let $X$ be a quantum set, viewed as a \Cstar-algebra $X \subseteq B(H)$. A \emph{projective subset in the sense of Weaver} is a right-module of $B(H, \C)$ over the commutant of $X$, that is a subspace $Y \subseteq B(H, \C)$ satisfying $YX' \subseteq Y$. 
\end{definition}

Just as for Weaver's binary relations, one can show that this definition is independent of the embedding of $X$ in $B(H)$. This can be proved directly, but it will also follow from the equivalence to our previous definitions, which do not depend on considering $X$ as a subalgebra of $B(H)$. For now, assume again that $X = \bigoplus_{i = 1}^k M_{n_i}$, $H = \bigoplus_{i = 1}^k \C^{n_i}$, and consequently $X' \cong \C^k$.
By the Riesz representation theorem, we can identify $B(H, \C) \cong H$, in which case the action of some $x \in X'$ is given by blockwise scalar multiplication by the conjugate entries of $x$. Since $X'$ is closed under entrywise conjugate, closure under this action is the same as closure under ordinary scalar multiplication. Finally, viewing $H$ as the diagonal fragment of $B(H)$ allows us to invoke \cref{lem:modproj}, so we may represent a projective subset $Y$ in the sense of Weaver graphically by the following $2$-morphism $P_Y$ in $\TFdHilb$.
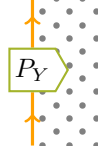
\begin{equation*}
  \input{diagrams/pss-modproj.tikz}
\end{equation*}
We now proceed analogously as above. We recover a projector on $X$ by composing on the right with $H^*$
\begin{equation}\label{eq:modtoregproj}
  \input{diagrams/pss-mod-doublewire.tikz}\enspace,
\end{equation}
and it is easy to see that the resulting operator satisfies the module condition of \cref{def:modcond},
\begin{equation*}
  \input{diagrams/pss-mod-condition.tikz}\enspace.
\end{equation*}
Conversely, we may connect the second and third legs by means of the duality morphism for $H$ to show that every projector that satisfies the module condition also satisfies
\begin{equation*}
  \input{diagrams/pss-mod-trace.tikz}\enspace.
\end{equation*}
We hence recover a projector onto a right-module of $H$ over the commutant by setting
\begin{equation*}
  \input{diagrams/pss-mod-correspond.tikz}\enspace.
\end{equation*}
Crucially, one can see that up to realignment $P_Y$ is nothing but the projection $\pproj{Y}$! Indeed, letting
\begin{equation*}
  \input{diagrams/pss-pproj-def.tikz}\enspace,
\end{equation*}
we see that left multiplication by $\pproj{Y}$ gives 
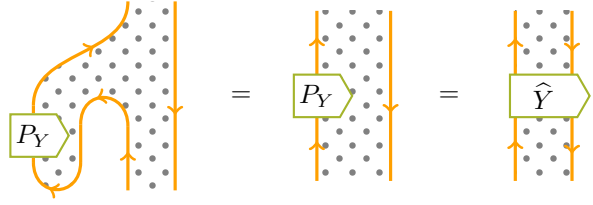
\begin{equation*}
  \input{diagrams/pss-leftmul.tikz}
\end{equation*}
by isomorphism of diagrams. For simplicity, we will denote the both the projection $\pproj{Y} \in X$ and the map $P_Y\colon H \to H$ by $\pproj{Y}$. In particular, it now makes sense to talk about the \emph{image} $\img \pproj{Y} \subseteq H$, which will simplify some of our later definitions. It will always be clear from the context which version of the map is meant.

To conclude this section, it will be useful to consider one last perspective on projective subsets. Namely, given a projective subset $Y \psubseteq X$ we may consider the image of $\mmap{Y}$. We have seen in \cref{lem:porpionbij} that $\mmap{Y}$ is of the form $L_p$ for some projection $p$, so we have $\img \mmap{Y} = pX$. These subspaces are obviously right-ideals, and it is well-known that every right-ideal on a quantum set is of this form. We thus conclude the following.

\begin{proposition}
  Let $X$ be a quantum set. There is a one-to-one correspondence between projective subsets $Y \psubseteq X$ and right-ideals in $X$. 
\end{proposition}

We will denote the ideal corresponding to a projective subset $Y$ by $\ideal{Y}$.\footnote{Another useful mnemonic to distinguish this notation from those for module homomorphisms and projections is to think of the circle as the dot of the \emph{i} in \emph{i}deal.} Combining all these different points of view yields a proposition by Weaver \cite[Proposition 2.23]{weaver_quantum_2010}, which shows that there is a one-to-one correspondence between binary quantum relations on $X$, one-sided ideals on $X \otimes X\op$, and projections in $X \otimes X\op$; closing the loop.

\subsection{Operations}

Of course, when we are talking about sets, we would like to define operations on them. In particular, we would like to form complements, intersections, and unions. It will be easiest to define these in terms of projections.

\begin{definition}\label{def:ops}
  Let $X$ be a quantum set and $Y, Z \psubseteq X$. 
  \begin{itemize}
    \item The \emph{complement} $\compl{Y}$ of $Y$ is given by $\pproj{\compl{Y}} = \identity - \pproj{Y}$.
    \item The \emph{intersection} $Y \cap Z$ is given by the projection $\pproj{Y \cap Z}$ onto $\img \pproj{Y} \cap \img \pproj{Z}$.
    \item The \emph{union} $Y \cup Z$ is given by the projection $\pproj{Y \cup Z}$ onto $\lspan\;(\img \pproj{Y} \cup \img \pproj{Z})$.
  \end{itemize}
\end{definition}

The \emph{empty projective subset} $\varnothing$ is thus given by the zero-projection, which makes it the unit for taking unions. Similarly, the \emph{complete projective subset}, the quantum set $X$ itself, is given by the identity projection, which makes it the unit for taking intersections. It remains to show that the resulting projections indeed give projective subsets of the correct quantum set. Concretely, we have to show that if $\pproj{Y}, \pproj{Z}$ are right-module homomorphisms over $X'$, then so are 
the constructed projections.
But by \cref{lem:modproj}, it suffices to show that the corresponding images are right modules over $X'$, which is routinely verified.
It is also readily shown that the above operations satisfy De Morgan's laws.

\begin{proposition}
  Let $X$ be a quantum set and $Y, Z \psubseteq X$. Then $\compl{(Y \cup Z)} = \compl{Y} \cap \compl{Z}$ and $\compl{(Y \cap Z)} = \compl{Y} \cup \compl{Z}$.
\end{proposition}

Using projections, we can also define the size of a projective subset.

\begin{definition}
  Let $X$ be a quantum set and $Y \psubseteq X$. The \emph{cardinality}, or \emph{size} of $Y$ is given by $\abs{Y} = \rk \pproj{Y}$.
\end{definition}

This definition diverges from how cardinality is usually defined for quantum sets. Given a quantum set $X$, we usually define its cardinality as the dimension of its underlying \Cstar-algebra. For instance, the cardinality of $M_n$ is $n^2$. On the other hand, we have seen that a quantum set is in particular a projective subset of itself, given by $\pproj{X} = \identity$. This in turn implies $\abs{M_n} = n \neq n^2$. For this reason, the above definition should perhaps be called ``subset cardinality'' instead. Since in the following we will always talk about subset cardinality and never about quantum set cardinality, we refrain from using this phrasing. 

\begin{remark}
We remark that the discrepancy can be resolved by instead defining the cardinality of a subset $Y \psubseteq X$ as $\rk \mmap{Y}$. Almost all of our results stay valid under this alternative definition; the exceptions being \cref{prop:chromnumineq,prop:indepsetweaversep,prop:cliqueweaversep}. The results stay conceptually correct, but the numeric statements would have to be restated. We chose the current definition because it is more natural in certain contexts. For example, the smallest non-zero cardinality of a projective subset is always $1$, while with the alternative definition it depends on the ambient quantum set. On $M_n$, for instance, the smallest non-zero cardinality would be $n$. Our definition also corresponds to how Weaver defines the sizes of independent sets and cliques. 
\end{remark}

Projective subsets can be contained in each other, so we need to extend our subset relation. It is hard not to verify that the following definition satisfies all the properties one would ask from such a relation, like $\varnothing \psubseteq Y$, $Y \cap Z \psubseteq Y$, and $Y \psubseteq Y \cup Z$.

\begin{definition}
  Let $X$ be a quantum set and $Y, Z \psubseteq X$. We let $Y \psubseteq Z$ if $\img \pproj{Y} \subseteq \img \pproj{Z}$.
\end{definition}

Since the ambient quantum set $X$ corresponds to the identity projections, the notions of projective subsets of quantum sets and projective subsets of other projective subsets coincide. Moreover, this definition interacts nicely with the standard notion of subsets of a quantum set. We have seen that taking a (standard) subset of a quantum set $X$ amounts to restricting to a subset of the summands in the direct sum decomposition. A projective subset on the other hand corresponds to a projection in $X$, which decomposes into projections on the individual summands. By choosing some of these projections to be the identity and the others to be the zero-map, we recover precisely the standard notion of subsets. 

\begin{proposition}\label{prop:ssetsarepsets}
  Let $X, Y$ be quantum sets. If $Y \subseteq X$ then $Y \psubseteq X$. 
\end{proposition}

Perhaps unsurprisingly, just as there is a natural correspondence between classical sets and their operations on the one hand, and classical (propositional) logic and its logical connectives on the other, there is a natural correspondence between projective subsets and their operations, and Birkhoff and von Neumann's \emph{quantum logic} \cite{birkhoff1975logic}. 
Quantum logic in its original form is quite controversial \cite{heunen2009topos}\cite[Chapter 17]{girard_blind_2011}, so will not got into detail about it here. However, it is well-known that quantum logic differs from classical logic in that it is not \emph{distributive}. 
For us, this leads to difficulties when trying to define a robust notion of disjointness: It is tempting to call two projective subsets $X, Y$ disjoint if $X \cap Y = \varnothing$. However, a desirable property of disjointness is that if $X$ and $Y$ are disjoint, and $X$ and $Z$ are disjoint, that then also $X$ and $Y \cup Z$ are disjoint. Classically, this property is precisely guaranteed by distributivity of $\cap$ over $\cup$, and indeed it fails in the quantum case: Consider the projective subsets $X, Y, Z \psubseteq M_2$ given by the projections onto distinct lines through the origin in $\C^2$. Then $X \cap Y = \varnothing$ and $X \cap Y = \varnothing$, but $Y \cup Z = M_2$ and thus $X \cap (Y \cup Z) = X \neq \varnothing$. Instead, we will define disjointness as \emph{orthogonality} of the corresponding projections.

\begin{definition}\label{def:disjointness}
  Let $X$ be a quantum set. Two projective subsets $Y, Z \psubseteq X$ are called \emph{disjoint} if $\img \pproj{Y} \bot \img \pproj{Z}$.
\end{definition}

Note that this definition has the abovementioned property. In general, it satisfies the following desirable statements, which are easy to verify.

\begin{observation}
  Let $X, Y, Z$ be projective subsets of some quantum set. Then
  \begin{enumerate}
    \item $X$ and $\compl{X}$ are disjoint.
    \item If $X \psubseteq Y$, then $X$ and $\compl{Y}$ are disjoint.
    \item If $X, Y$ are disjoint and $X, Z$ are disjoint, then $X, Y \cup Z$ are disjoint.
  \end{enumerate}
\end{observation}

The final operation on projective subsets we will consider is that of the cartesian product. Let $X$ be quantum set and $Y, Z \psubseteq X$. An intuitive choice would be to define $Y \times Z$ as the projective subset of $X \otimes X$ given by the projection $\pproj{Y} \otimes \pproj{Z}$. However, as we have mentioned in \cref{ssec:anyothername}, classical sets of the form $Y \times Z$ can be seen as binary relations. Similarly, we would like that the cartesian product of projective subsets is similarly a binary relation as defined by Musto, Reutter, and Verdon. These are not projective subsets of $X \otimes X$, but of $X\op \otimes X$. This motivates the following definition.

\begin{definition}\label{def:cartprod}
  Let $X, X'$ be quantum sets and $Y \psubseteq X, Z \psubseteq X'$. The \emph{cartesian product} $Y \times Z$ is the projective subset of $X\op \otimes X'$ given by the projection $\pproj{Y}\transpose \otimes \pproj{Z} = \overline{\pproj{Y}} \otimes \pproj{Z}$.
\end{definition}

It is also easily seen in the graphical calculus that taking the conjugate on the left side recovers Weavers bimodules over the commutant. Recall that taking the conjugate corresponds to mirroring the diagram across a vertical axis. $\pproj{Y \times Z}$ is thus given by
\begin{equation*}
  \input{diagrams/y-cprod-z.tikz}\enspace.
\end{equation*}
  Note that under these definitions we have $E(G) \psubseteq V(G) \times V(G)$, exactly the same as for classical graphs!

\subsection{Projective Subsets and Functions}

Before we move on to studying quantum graph properties, we take a brief detour through the interplay between classical functions and projective subsets. It is a foundational concept of discrete mathematics how sets interact with functions. For example, functions have \emph{images}, which are subsets of their codomain, and they can be \emph{restricted} to subsets of their domain. These two notions are readily generalised to ordinary subsets of quantum sets.

\begin{definition}
  Let $f\colon X \to Y$ be a classical function between quantum sets. The \emph{image} of $f$, denoted $\img f$ is the image of the underlying linear map of $f$.
\end{definition}

\begin{definition}\label{def:cfuncrestr}
  Let $X, Y, Z$ be quantum sets such that $Y \subseteq X$ and let $f\colon X \to Z$ be a classical function. The \emph{restriction} of $f$ to $Y$, denoted $f\restrict{Y}$ is defined as the restriction of the underlying linear map of $f$.
\end{definition}

These definitions are well-behaved in the sense that the following two results hold, which we state here without proof. They follow from an algebraic argument, noting that the adjoint of a classical function is a $*$-homomorphism between direct sums of simple algebras. 

\begin{proposition}\label{prop:cfuncimgissset}
  Let $X, Y$ be quantum sets and $f \colon X \to Y$ a classical function. Then $\img f$ is a quantum set, and $\img f \subseteq Y$.
\end{proposition}

\begin{proposition}
  Let $X, Y, Z$ be quantum sets such that $Y \subseteq X$ and let $f\colon X \to Z$ be a classical function. Then $f\restrict{Y}$ is a classical function $Y \to Z$.
\end{proposition}

Coveniently, these definitions may be refined to projective subsets. Some care is necessary when defining the restriction of a classical function to a projective subset, since classical functions are defined as morphisms between $\dagger$-SSFMs. Instead we define the restriction as a suitable equivalence class of classical functions. To improve readability, we denote the underlying linear map of a classical function $f$ by $\ucl{f}$.

\begin{definition}
  Let $X, Y, Z$ be quantum sets such that $Y \psubseteq X$ and let $f\colon X \to Z$ be a classical function. The \emph{restriction} of $f$ to $Y$, denoted $f\restrict{Y}$ is given by the set of classical functions $\{f' \colon Y' \to Z \mid \uclarg{f'}(\ideal{Y}) = \uclarg{f}(\ideal{Y}) \}$, where $Y' \subseteq X$ is minimal such that $Y \psubseteq Y'$. Its \emph{image} is given by $\img f\restrict{Y} = \uclarg{f}(\ideal{Y})$.
\end{definition}

We also write $f(Y)$ for $\img f\restrict{Y}$, mirroring the classical notation. The definition extends \cref{def:cfuncrestr} in the following sense: If $Y$ is a quantum set and hence $Y \subseteq X$, then $Y' = Y$ and $f\restrict{Y}$ is the singleton set of containing only $\ucl{f}\restrict{Y}$.
It is an immediate consequence of \cref{prop:cfuncimgissset,prop:ssetsarepsets} that the image of a classical function is a projective subset. The following lemma implies that the same holds for the image of a restriction.

\begin{lemma}
  Let $X, Y$ be quantum sets, $X' \psubseteq X$, and $f\colon X \to Y$ a classical function. Then $\uclarg{f}(\ideal{X'})$ is a right-ideal in $Y$.
  \begin{proof}
    We show this graphically. Let $x \in X'$ and $y \in Y$ be arbitrary. Then we have
    \begin{equation*}
      \input{diagrams/f-img-ideal.tikz}\enspace,
    \end{equation*}
    where $x' \in X'$.
  \end{proof}
\end{lemma}

\begin{corollary}\label{cor:psetpushforward}
  Let $X, Y, Z$ be quantum sets such that $Y \psubseteq X$ and let $f\colon X \to Z$ be a classical function. Then $f(Y) \psubseteq Z$.
\end{corollary}

We conclude this excursion on classical functions by defining their cartesian product. Classically, if $f \colon X \to A$ and $g \colon Y \to B$ are functions, then $f \times g \colon X \times Y \to A \times B$ is defined by $(f \times g)(v, x) = (f(v), g(v))$. We generalise this definition in the intuitive way.

\begin{definition}
  Let $A, B, X, Y$ be quantum sets and $f\colon X \to A$, $g\colon Y \to B$ be classical functions. Then their \emph{cartesian product} $f\times g$ is given by the tensor product of their underlying linear maps, $\ucl{(f\times g)} = \ucl{f} \otimes \ucl{g}$.
\end{definition}

It is not hard to verify that the resulting linear map is a classical function $X \otimes Y \to A \otimes B$. It is also a classical function $X\op \otimes Y \to A\op \otimes B$ and hence under \cref{def:cartprod} we may write $f\times g\colon X \times Y \to A \times B$. It follows that this definition is nicely compatible with our characterisation of binary relations between quantum sets $X, Y$ as projective subsets of $X \times Y$. For example, just as one would do classically, we may define the \emph{pushforward} of a binary relation $R \psubseteq X \times Y$ along $f\times g$ as $(f \times g)(R)$, which by \cref{cor:psetpushforward} is a binary relation between $A$ and $B$.

\section{Quantum Graph Properties via Projective Subsets}\label{sec:props}
Having established some fundamental properties of projective subsets, we now want to turn our attention to quantum graphs and their properties. Concretely, we are mainly interested in connected components, colourings, cliques, and independent sets. We will investigate these concepts individually; first introducing a definition in terms of projective subsets, and then comparing it with existing definitions in the literature. However, before we look at these concrete properties let us first phrase some fundamental properties of quantum graphs in terms of projective subsets.

\subsection{Basic Properties}

Recall that we generally distinguish between \emph{loopless} quantum graphs and those with loops; a special case of the latter being quantum graphs with \emph{loops at every vertex}. Concretely, a quantum graph $G$ with adjacency operator $A_G$ is said to be loopless or to have loops at every vertex if respectively
\begin{equation}\label{eq:loopcond2}
  \input{diagrams/adj-op-irreflexive.tikz} \qquad\quad\text{or}\quad\qquad \input{diagrams/adj-op-reflexive.tikz}\enspace.
\end{equation}

For classical graphs, this definition says that the Schur product of the adjacency matrix with the identity is either $0$, in the case of no loops, or the identity, in the case of loops at every vertex. However, for classical graphs there is an equivalent definition in terms of subsets.

\begin{observation}
  Let $G$ be a classical graph. Then $G$ is loopless if $\Delta$ and $E(G)$ are disjoint. It has loops at every vertex if $\Delta \subseteq E(G)$.
\end{observation}

The $\Delta$ here denotes the \emph{diagonal relation} on $V(G)$, defined as $\Delta \coloneqq \{(v, v) \mid v \in V(G)\} \subseteq V(G) \times V(G)$. It turns out that this alternative definition of the existence of loops can be generalised to the quantum case. To that end, we need an analogue of the diagonal relation on a quantum set. Fortunately, multiple equivalent definitions of the diagonal relation have been proposed \cite{weaver_quantum_2021,kornell_discrete_2023}. Recall that a binary relation on a quantum set $X$ in the sense of Weaver is a bimodule over the commutant $X'$. Weaver then defines the diagonal relation as $X'$ itself \cite[Definition 4.2]{weaver_quantum_2021}. We know that the commutant of a quantum set $X$ is given by the direct sum of multiples of the identity, so we may express the corresponding projector graphically as 
\begin{equation*}
  \input{diagrams/delta-fullproj.tikz} 
\end{equation*}
We can again view $B(H)$ as $H^* \otimes H$, which allows us to write it as
\begin{equation*}
  \input{diagrams/delta-proj-doublewire.tikz}
\end{equation*}
Composing with $H$ and $H^*$ wires to obtain $\mmap{\Delta}$ yields
\begin{equation*}
  \input{diagrams/delta-proj-bimod.tikz}.
\end{equation*}
This is exactly the \emph{Frobenius map}, which has also been called the \emph{sharing} map \cite{goldberg_quantum_2026}. Note that for classical sets $X = \C^n$, this is precisely the projector onto $\lspan \{\ket{i} \otimes \ket{i} \mid i \in [n]\}$. 

Now recall that the loop conditions of \cref{eq:loopcond2} can equivalently be phrased in terms of the edge relation $E(G)$. A quantum graph has no loops or loops at every vertex if respectively
\begin{equation*}  
  \input{diagrams/edgerel-irreflexive.tikz} \qquad\quad\text{or}\qquad\quad\enspace \input{diagrams/edgerel-reflexive.tikz}\enspace.
\end{equation*}
From our observations in \cref{ssec:anyothername}, we know that these conditions are really of the form
\begin{equation*}
  \input{diagrams/edgerel-irreflexive-doublewire.tikz}  \qquad\quad\text{and}\qquad\quad\enspace \input{diagrams/edgerel-reflexive-doublewire.tikz}\enspace.
\end{equation*}
We can bend down the outer wires to obtain the equivalent conditions
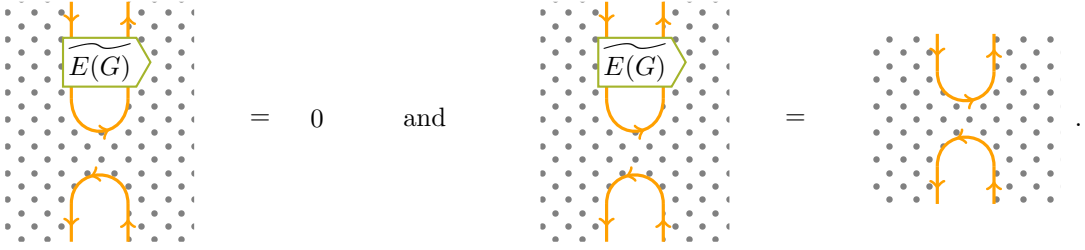
\begin{equation*}
  \input{diagrams/edgerel-irreflexive-doublewire-prec.tikz} \qquad\quad\text{and}\qquad\quad\enspace \input{diagrams/edgerel-reflexive-doublewire-prec.tikz}\enspace.
\end{equation*}
But the first condition just says that $\img \pproj{E(G)}\, \bot \img \pproj{\Delta}$, while the second condition says that $\img \pproj{\Delta} \subseteq \img \pproj{E(G)}$. These are precisely our definitions of disjointness and subset containment of projective subsets. We thus obtain the following statement.

\begin{proposition}\label{prop:loopsanddiagrel}
  A quantum graph $G$ is loopless if and only if $E(G)$ and $\Delta$ are disjoint. It has loops at every vertex if $\Delta \psubseteq E(G)$.
\end{proposition}

In fact, Weaver requires precisely that $\img \pproj{\Delta} \subseteq \img \pproj{E(G)}$ in his definition of quantum graphs \cite[Definitions 4.2 and 4.5]{weaver_quantum_2021}. Musto, Reutter, and Verdon then proved that this property corresponds precisely to having loops at every vertex \cite[Theorem VII.7]{musto_compositional_2018}, which yields an alternative proof of the second part of \cref{prop:loopsanddiagrel}. The result also allows us to add and remove loops analogously to the classical case.

\begin{corollary}\label{cor:loopaddremove}
  Let $G$ be a quantum graph potentially containing loops. Then 
  \begin{itemize}
    \item $\loops{G}$ with $E(\loops{G}) = E(G) \cup \Delta$ is a quantum graph with loops at every vertex.
    \item $\looprm{G}$ with $E(\looprm{G}) = E(G) \setminus \Delta$ is a loopless quantum graph.
  \end{itemize}
\end{corollary}

If $G$ was loopless, then this construction corresponds precisely to adding the identity to the adjacency operator. Conversely, if $G$ had loops at every vertex then the construction amounts to subtracting the identity from the adjacency operator. The proof is a straightforward diagrammatic derivation using that in these cases it holds that $\mmap{E(G) \cup \Delta} = \mmap{E(G)} + \mmap{\Delta}$ and $\mmap{E(G) \setminus \Delta} = \mmap{E(G)} - \mmap{\Delta}$. We omit it here.

This treatment of loops also allows us to nicely phrase the concept of \emph{graph complements}. It is tempting to simply define the complement of a quantum graph $G$ as the quantum graph $\compl{G}$ on the same vertex set with $E(\compl{G}) = \compl{E(G)}$. However, if $G$ was loopless then the resulting graph will have loops at every vertex. In classical graph theory, one usually considers complements without loops, so we will follow this convention.

\begin{definition}
  Let $G$ be a quantum graph. The \emph{complement} $\compl{G}$ of $G$ is the quantum graph given by $V(\compl{G}) = V(G)$ and $E(\compl{G}) = \compl{E(G)} \setminus \Delta$.
\end{definition}

One may verify that the complement operation is an involution on loopless quantum graphs, that is it satisfies $\compl{(\compl{G})} = G$.

\subsection{Connected Components}\label{ssec:conncomp}
Let us now see how we can use projective subsets to express more complex quantum graph structure.
A \emph{connected component} of a classical graph $G$ is a subset of the vertex set that has no edges leaving it. A decomposition of $G$ into connected components is then a partition of the vertex set in which every block is a connected component, that is there exist no edges between vertices belonging to different blocks. This notion has a straightforward formalisation. Given two blocks $X_1, X_2 \subseteq V(G)$ of the partition, we can form the complete bipartite graph $G_{12}$ with vertex set $X_1 \sqcup X_2$. This is precisely the graph that contains all possible edges between $X_1$ and $X_2$. The partition is thus a valid decomposition into connected components if and only if $E(G)$ and $E(G_{12})$ are disjoint for all choices of $X_1 \neq X_2$.

We almost have everything we need to generalise this definition to the quantum case. The edge relation of the classical complete bipartite graph with vertex set $X_1 \sqcup X_2$ is given by $X_1 \times X_2$, and we have defined the cartesian product of projective subsets in \cref{def:cartprod}. Similarly, we have defined what it means for projective subsets to be disjoint in \cref{def:disjointness}. The only missing definition is that of a partition of a quantum set, which we define completely analogously to the classical case.

\begin{definition}
  Let $X$ be a quantum set. A \emph{partition} of $X$ is a collection of mutually disjoint projective subsets $P_1, \dots, P_k \psubseteq X$, such that $\bigcup_s P_s = X$.
\end{definition}

We can now state our definition of decompositions of a quantum graph into connected components.

\begin{definition}\label{def:conncomps}
  A quantum graph $G$ admits a decomposition into $k$ \emph{connected components} if there exists a partition of $V(G)$ into projective subsets $P_1, \dots, P_k$, such that $E(G)$ and $P_s \times P_t$ are disjoint for all $s \neq t \in [k]$.
\end{definition}

Definitions for connected components of quantum graphs in the literature have mostly been restricted to the binary question of whether a quantum graph is \emph{connected} \cite{dominguez_connectivity_2019,matsuda_algebraic_2024,courtney_connectivity_2025}. It is easy to see that a classical graph is connected if and only if there exists no nontrivial connected component. \cref{def:conncomps} is thus easily adapted to this simplified case.

\begin{definition}\label{def:connected}
  A quantum graph $G$ is \emph{disconnected} if there exists a non-trivial projective subset $P \psubseteq V(G)$ such that $E(G)$ and $P \times \compl{P}$ are disjoint. It is called \emph{connected} if it is not disconnected.
\end{definition}

Let us now contrast our definition with those found in the literature. The first definition of connectedness that we will consider is of a similar graph-theoretic flavor as ours. The idea is to define connectedness through \emph{graph homomorphisms}. Classically, a graph homomorphism $f \colon G \to H$ is a mapping from $V(G)$ to $V(H)$ that preserves vertex adjacency: If $v, w \in V(G)$ with $v \sim w$, then also $f(v) \sim f(w)$. Observe that as a consequence, graph homomorphisms also have to preserve connectedness. Indeed, if $v$ and $w$ are connected by a path $v = u_1, u_2, \dots, u_k = w$ in $G$, then $f(u_1), \dots, f(u_k)$ must be a path in $H$ (potentially repeating vertices). It follows that a classical graph $G$ is disconnected if and only if there exists a surjective homomorphism into $K^{c, \circ}_2$, the disjoint union of two vertices with self-loops. The surjectivity requirement forces the existence of $v, w \in V(G)$ that get mapped to different vertices in $K^{c, \circ}_2$. By our observations, $v$ and $w$ then cannot be connected by a path in $G$. There is a notion of graph homomorphisms between quantum graphs, which allows to take this classical result and use it to define connectedness for quantum graphs. This approach was taken by Matsuda \cite{matsuda_algebraic_2024}.

\begin{definition}[cf. Proposition V.3 in \cite{musto_compositional_2018}]\label{def:qghom}
Let $G$ and $H$ be quantum graphs. A \emph{graph homomorphism} $G \to H$ is a classical function $\varphi \colon V(G) \to V(H)$ satisfying 
  \begin{equation*}
    \input{diagrams/graph-homomorphism.tikz}\enspace.
  \end{equation*}
\end{definition}

Observe that this definition says nothing but $(\varphi \times \varphi)(E(G)) \psubseteq E(H)$, which is precisely how one may define homomorphisms $G \to H$ between classical graphs. 

\begin{definition}
  A quantum graph $G$ is \emph{connected in the sense of Matsuda} if there does not exist a surjective graph homomorphism $G \to K^{c, \circ}_2$.
\end{definition}

Courtney, Ganesan, and Wasilewksi \cite{courtney_connectivity_2025} define a quantum graph to be disconnected if there exists a non-trivial projection $p \in V(G)$ such that $L_p A_G = A_G L_p$. First note that this is equivalent to requiring $L_p A_G L_{p^\bot} = 0$. Indeed, multiplying the former statement from the left with $L_p$ and from the right with $L_{p^\bot}$ yields $L_p L_p A_G L_{p^\bot} = A_G L_p L_{p^\bot}$ and thus $L_p A_G L_{p^\bot} = 0$. Conversely, we have $L_{p^\bot} = \identity - L_p$, so the latter statement is equivalent to $L_p A_G - L_p A_G L_p = 0$. Since for undirected quantum graphs $A_G$ is self-adjoint, we then get 
\begin{equation*}
  L_p A_G = L_p A_G L_p = (L_p A_G L_p)\adjoint = (L_p A_G)\adjoint = A_G L_p
\end{equation*}
as desired. 

\begin{definition}\label{def:cgwconnected}
  A quantum graph $G$ with adjacency operator $A_G$ is \emph{CGW connected} if there does not exist a non-trivial projection $p \in V(G)$ such that $L_p A_G L_{p^\bot} = 0$.
\end{definition}

This definition is motivated by the fact that if $G$ is a classical disconnected graph then its adjacency matrix can be taken to be block diagonal. Since $p$ is a projection, $L_p$ is a projector which then projects onto one of these blocks. Sandwiching the adjacency operator by $L_p$ and its orthogonal complement thus extracts an off-diagonal block, which should be zero. This phrasing immediately raises the question why one should not permit arbitrary projections $P \colon V(G) \to V(G)$ instead of only those arising as left-multiplication maps. In \cite{courtney_connectivity_2025}, the authors note that one can always choose $P$ to be the projector onto an eigenspace of $A_G$, in which case the equation $(\identity - P) A_G P = 0$ will always be satisfied. For classical graphs, the eigenspaces of the adjacency matrix are not necessarily the connected components, so if one wishes to recover the original notion in the classical case one has to restrict the choice of projectors somehow. The authors then simply say that the projections in $V(G)$ happen to be a good candidate for this choice. We can now answer \emph{why} they are a good choice: They correspond precisely to subsets of the vertex set!

The two definitions of connectedness we have considered so far are nicely graph theoretical. However, they were preceded by a more operator-algebraic definition, which we mention here because of its surprisingly close connection to projective subsets. We state the definition first and then motivate it.

\begin{definition}[cf. Definition 3.1 in \cite{dominguez_connectivity_2019}]\label{def:cdsconnect}
  Let $G$ be a quantum graph with corresponding operator space $S \subseteq M_n$. Then $G$ is \emph{connected in the sense of Chávez-Domínguez and Swift} if there exists an $m \in N$ such that $(S \oplus \identity)^m = M_n$.
\end{definition}

The idea of viewing a quantum graph as a subspace of $M_n$ stems from the foundational paper by Duan, Severini, and Winter \cite{duan_zeroerror_2013} where they generalise the concept of Shannon's confusability graphs \cite{shannon2003zero} to quantum channels. Concretely, they associate to a quantum channel $\Phi$ with Kraus operators $F_1, \dots, F_k$ the operator space $S_\Phi = \{F_iF_j\adjoint \mid i, j \in [k]\} \subseteq M_n$. They then note that this operator space behaves in some sense like a classical (confusability) graph. This connection to graphs can be made precise: Operator spaces correspond precisely to edge relations of quantum graphs on $M_n$. Indeed, note that for $H = \C^n$ a subspace of $M_n = B(H)$ is nothing but a bimodule of $B(H)$ over the commutant $M_n' = \C\identity$---a binary relation on $M_n$. Note that quantum graphs arising this way from quantum channels always have loops at every vertex. This follows from the trace-preservation condition $\sum_i F_iF_i\adjoint = \identity$. 

Classical graphs fit into the operator space framework by associating to a classical graph $G$ the operator space 
\begin{equation}\label{eq:clgraphopspace}
  S_G \coloneqq \{\ket{i}\bra{j} \mid i, j \in [n], i \sim j \}.
\end{equation}
Defining for operator spaces $S, T \subseteq M_n$
\begin{equation*}
  ST \coloneqq \lspan\{st \mid s \in S, t \in T\}, \qquad\quad S^0 = \C\identity, \qquad\quad S^m = SS^{m-1},
\end{equation*}
we have $(S_G \oplus \identity)^m = M_n$ if and only if for every pair of vertices $i, j \in [n]$ there exists a path from $i$ to $j$ of length at most $m$. The idea is that if $i$ and $j$ are connected by a path through vertices $\ell_1, \dots, \ell_{m-1}$, then $\ket{i}\bra{j} = \ket{i}\bra{\ell_1} \ket{\ell_1} \bra{\ell_2} \dots \ket{\ell_{m-1}}\bra{j} \in S^m$. Adding the identity allows us to pad shorter paths, guaranteeing that $S^{k-1} \subseteq S^k$.\footnote{The authors of \cite{dominguez_connectivity_2019} do not add the identity to the operator space. This is because they assume that quantum graphs have loops at every vertex, which means that the corresponding operator spaces already contain the identity.}

We have seen how we can associate an operator space to a classical graph. It remains to formally define what we mean by ``corresponding operator space'' for arbitrary quantum vertex sets. To that end, note that every subspace of $M_n$ is a bimodule over $X'$ for some quantum set $X$ and thus corresponds to a binary relation on $X$. On the other hand, we have seen in \cref{ssec:anyothername} that given binary relation $R \psubseteq X \times X$, the projector onto the corresponding operator bimodule is given by $\pproj{R}$. We generalise the concept through precisely this correspondence.

\begin{definition}
  Let $G$ be a quantum graph. Its \emph{corresponding operator space} $S$ is given by $S = \img \pproj{E(G)}$.
\end{definition}

Note that this perfectly recovers the representation of classical graphs introduced in \cref{eq:clgraphopspace}. Indeed, by \cref{prop:adjopedgerelequiv}, the edge relation of a classical graph is given by the following projection $\mmap{E(G)} \colon \C^n \otimes \C^n \to \C^n \otimes \C^n$.
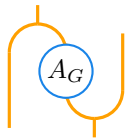
\begin{equation*}
  \input{diagrams/edgerel-from-adjop.tikz}
\end{equation*}
To obtain $\pproj{E(G)}$, we write $\mmap{E(G)}$ in the double-wire calculus of $\TFdHilb$ and trace out the outer wires, which yields
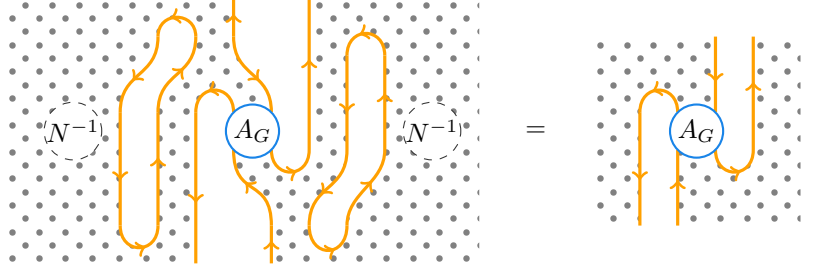
\begin{equation*}
  \input{diagrams/edgerel-from-adjop-doublewire.tikz}\enspace.
\end{equation*}
Recall that the composition of the $1$-morphisms $\C^n \colon \C \to \C^n$ and $(\C^n)^* \colon \C^n \to \C$ along $\C^n$ just fuses the corresponding wires into one. The above $2$-morphism is thus equal to
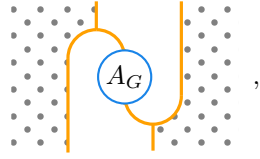
\begin{equation*}
  \input{diagrams/edgerel-from-adjop-doublewire-glued.tikz}\enspace,
\end{equation*}
and it is straightforward to check that it projects onto the operator system $S_G$ of \cref{eq:clgraphopspace}. Note that if $G$ is undirected and $i \sim j$, then $S_G$ contains both $\ket{i}\bra{j}$ and $\ket{j}\bra{i}$. It follows that $S_G$ is closed under adjoints. This is not a coincidence: By a graphical argument that we omit here, one can show that a quantum graph is undirected if and only if its corresponding operator space is closed under adjoints, see Section VII of \cite{musto_compositional_2018} for more details. This matches Weaver's definition of symmetric relations \cite[Definition 4.2]{weaver_quantum_2021}.

\begin{remarkbox}
Observe that the correspondence between quantum graphs and their operator spaces is not bijective. For instance, any operator bimodule over $X'$ may also be seen as an operator bimodule over $M_n'$ and thus as a quantum graph on $M_n$. 
An example of this is the operator system associated to a classical graph. The operator space in \cref{eq:clgraphopspace} is a bimodule over $(\C^n)'$, but it we may also view it as a bimodule over $M_n'$. This yields a quantum graph over $M_n$, which was for example studied in \cite{spitzer_quantum_2026}, where it is called $X_{G, \ccdot}$.
The quantum graphs resulting from this change in perspective are closely related, but will generally have slightly different properties by virtue of the ambient quantum sets admitting different projective subsets. Consider, for example, connectedness. Whether a quantum graph is connected depends on the existence of a projective subset of the vertex set that satisfies a disjointness condition with respect to the edge set $E(G)$. If two quantum graphs have the same operator space, the disjointness condition is the same for both, so the connectedness only depends on which projective subsets are admitted. Viewing a projective subset as a right-module over the commutant of the vertex set, one can see that every projective subset of $X$ is also a projective subset of $M_n$. It is thus strictly harder for the quantum graph on $M_n$ to be connected. Similar results will hold for the other properties we define.
\end{remarkbox}

In \cite{courtney_connectivity_2025}, the authors already proved that all these definitions of connectedness are equivalent. In fact, their list contains even more equivalent statements that are very useful, but do not quite fit the graph-theoretic flavor of our treatment. We show that they are in turn equivalent to our definition in terms of projective subsets. We will need the following lemma, which will remain of central importance for the rest of the paper.

\begin{lemma}\label{lem:disjsandwich}
  Let $G$ be a quantum graph with adjacency operator $A_G$ and let $P, Q \psubseteq V(G)$. Then the following are equivalent.
  \begin{enumerate}
    \item $E(G)$ and $P \times Q$ are disjoint.
    \item $\mmap{P} A_G \mmap{Q} = 0$.
  \end{enumerate}
  \begin{proof}
    We first phrase the disjointness condition in terms of $\mmap{P}$ and $\mmap{Q}$. Per definition, $E(G)$ and $P \times Q$ are disjoint if and only if
    \begin{equation*}
      \input{diagrams/edgerel-disj.tikz}. 
    \end{equation*}
    By composing from the left and right with the suitable 1-morphisms, we obtain an equivalent statement for the module maps:
    \begin{equation}\label{eq:edgreldisjmmap}
      \input{diagrams/edgerel-disj-mmap.tikz}.
    \end{equation}
    Now recall the correspondence between adjacency operators and edge relations of quantum graphs. We have
    \begin{equation*}
      \input{diagrams/adjop-sandwich-equiv.tikz}\enspace.
    \end{equation*}
    The second equation follows from isomorphism of diagrams and the definition of the cap map in terms of multiplication and counit. The third equation follows from the bimodule property of $E(G)$ and unitality. The final equation follows from the left-module property of $\overline{\mmap{P}}$, the bimodule property of $E(G)$ and counitality.
    It is now obvious that if $E(G)$ and $P \times Q$ are disjoint, that is \cref{eq:edgreldisjmmap} holds, then $\mmap{P} A_G \mmap{Q} = 0$. Conversely, if 
    \begin{equation*}
      \input{diagrams/adjop-cond-to-edgerel-cond.tikz}
    \end{equation*}
    then also
    \begin{equation*}
      \input{diagrams/adjop-cond-to-edgerel-cond-2.tikz}.
    \end{equation*}
    But the left-hand side is precisely equal to the left-hand side of \cref{eq:edgreldisjmmap} by the (bi-, left-, right-) module properties of $E(G), \overline{\mmap{P}}, \mmap{Q}$ respectively. This completes the proof.
  \end{proof}
\end{lemma}

We can now extend Courtney, Ganesan, and Wasilewksi's equivalence results by our definition.

\begin{theorem}[cf. Theorem 3.4 in \cite{courtney_connectivity_2025}]\label{thm:connectedequiv}
  Let $G$ be a quantum graph with adjacency operator $A_G$ and corresponding operator space $S$. Then the following are equivalent.
  \begin{enumerate}
    \item $G$ is CGW connected: There exists a non-trivial projection $p \in V(G)$ such that $L_p A_G L_{p^\bot} = 0$.
    \item $G$ is connected in the sense of Chávez-Domínguez and Swift: $(S \oplus \identity)^m \neq M_n$ for all $m \in \N$. 
    \item $G$ is connected in the sense of Matsuda: There exists a surjective graph homomorphism $G \to \compl{K}_2$.
    \item $G$ is connected: There exists a non-trivial projective subset $P \psubseteq V(G)$ such that $E(G)$ and $P \times \compl{P}$ are disjoint.
  \end{enumerate}
  \begin{proof}
    The equivalence of items $1$-$3$ was already proved in \cite{courtney_connectivity_2025}. In their work, 2 was phrased as $\Cstar(S) \neq M_n$. This is easily seen to be equivalent, as the \Cstar-algebra generated by $S$ is precisely given by $\Cstar(S) = \lspan \{ S^m \mid m \geq 0 \}$. We show that $4 \iff 1$. Note that since $p \in V(G)$ is a projection, we have $L_p = \mmap{P}$ for some projective subset $P \psubseteq V(G)$. The condition $L_p A_G L_{p^\bot} = 0$ is thus equivalent to $\mmap{P} A_G \mmap{\compl{P}} = 0$, which by \cref{lem:disjsandwich} is equivalent to $E(G)$ and $P \times \compl{P}$ being disjoint.
  \end{proof}
\end{theorem}

In \cite{spitzer_quantum_2026}, Courtney, Ganesan, and Wasilewksi's definition was extended from connectedness to decompositions into connected components.

\begin{definition}[cf. Definition 2.14 in \cite{spitzer_quantum_2026}]\label{def:cgwdecomp}
  Let $G$ be a quantum graph. We say that $G$ admits a \emph{CGW decomposition into $k$ connected components} if there exist non-trivial projectors $p_1, \dots, p_k$ with $\sum_i p_i = \identity$ such that $L_{p_s} A_G L_{p_t} = 0$ for all $s \neq t \in [k]$.
\end{definition}

Note that for $k = 2$ the definition reduces to \cref{def:cgwconnected} of CGW connectedness, since in this case $p_2 = \identity - p_1 = p_1^\bot$. To show that CGW decompositions coincide with our decompositions, we need the following characterisation of partitions.

\begin{lemma}\label{lem:partchar}
  Let $X$ be a quantum set. A collection of projective subsets $P_1, \dots, P_k \psubseteq X$ forms a partition of $X$ if and only if $\sum_i \pproj{P_i} = \identity$.
  \begin{proof}
    We have $\sum_i \pproj{P_i} = \identity$ if and only if $\img \pproj{P_i} \bot \img \pproj{P_j}$ and $\lspan \bigcup_i \img \pproj{P_i} = X = \img \identity$, which by \cref{def:ops} and \cref{def:disjointness} are precisely the conditions $P_i$ and $P_j$ being disjoint and $\bigcup_i P_i = X$.
  \end{proof} 
\end{lemma}

The equivalence of CGW decompositions and our decompositions is now immediate.

\begin{theorem}\label{thm:conncompequiv}
  Let $G$ be a quantum graph. Then the following are equivalent.
  \begin{enumerate}
    \item $G$ admits a decomposition into $k$ connected components.
    \item $G$ admits a CGW decomposition into $k$ connected components.
  \end{enumerate}
  \begin{proof}
    The equivalence of $L_{p_s} A_G L_{p_t} = 0$ and the disjointness of $P_s \times P_t$ and $E(G)$ follows again directly from \cref{lem:disjsandwich}. On the other hand, by \cref{lem:partchar}, $p_1, \dots, p_k \in V(G)$ are non-trivial projections summing to the identity if and only if the projective subsets $P_1, \dots, P_k \psubseteq V(G)$ with $\pproj{P_i} = p_i$ for all $i \in [k]$ form a partition of $V(G)$. 
  \end{proof}
\end{theorem}

\subsection{Colouring}
Let us next turn our attention to colouring. A $k$-colouring of a classical graph $G$ is an assignment of one of $k$ colours to each vertex, such that adjacent vertices get different colours. We can again phrase this definition purely in terms of subsets. The idea is to consider the subsets $P_s$ of vertices having colour $s \in [k]$. These sets partition $V(G)$. The adjacency condition is then equivalent to $P_s \times P_s$ and $E(G)$ being disjoint. This is because $P_s \times P_s$ consists of all possible edges between vertices in $P_s$, and we require that none of them exist in $E(G)$. This leads to the following translation to the quantum case.

\begin{definition}\label{def:colouring}
  A quantum graph $G$ has a \emph{$k$-colouring} if there exists a partition of $V(G)$ into projective subsets $P_1, \dots, P_k$ such that $E(G)$ and $P_s \times P_s$ are disjoint for all $s \in [k]$. The smallest $k$ such that $G$ admits a $k$-colouring is called the \emph{chromatic number} and is denoted by $\chi(G)$.
\end{definition}

We may invoke \cref{lem:disjsandwich} to derive the following equivalent characterisation of colouring.

\begin{proposition}
  A quantum graph $G$ with adjacency operator $A_G$ has a $k$-colouring if and only if there exists a partition of $V(G)$ into projective subsets $P_1, \dots, P_k$ such that $\mmap{P_s} A_G \mmap{P_s} = 0$ for all $s \in [k]$.
\end{proposition}

This highlights a nice duality between colouring and decompositions into connected components. While the adjacency operator of a disconnected graph can be written as a block diagonal matrix, the adjacency operator of a $k$-colourable graph can be written as a ``block off-diagonal'' matrix, where the diagonal blocks are zero.

Just as for connected components, there is already an established notion of $k$-colourings which was introduced by Brannan, Ganesan, and Harris \cite{brannan2022quantum} in the context of non-local games. Classically, one may phrase the existence of a $k$-colouring in terms of the so called \emph{colouring game}. As usual in the setting of non-local games, two players, Alice and Bob, try to convince a referee that they know a $k$-colouring of the classical graph. The referee sends an arbitrary vertex of the classical graph to Alice and Bob respectively. The players then have to answer, without communicating, with one of $k$ colours. The players have a \emph{winning strategy} if they can guarantee that their answers are always compatible: If they receive the same vertex they have to answer with the same colour, and if they receive two adjacent vertices they have to answer with distinct colours. It is possible to show that Alice and Bob have a winning strategy if and only if the graph is $k$-colourable. 

\begin{remark}\label{rem:nonlocalstrats}
  Beyond these so called \emph{local} strategies, one is usually moreover interested in strategies where the players have access to some shared resource, like an entangled quantum state. This leads to relaxations like the quantum chromatic number. While these concepts can also be studied in the quantum setting, we are here only interested in ``classical'' properties of quantum graphs.
\end{remark}

In \cite{brannan2022quantum} the colouring game was generalised to quantum graphs. The main obstacle for this generalisation is that quantum graphs no longer have individual vertices. Instead the referee prepares a bipartite quantum state in a way that Alice and Bob respectively have access to one of the two systems. The players then perform a measurement on their respective systems to determine their (classical) answer. The classical compatibility conditions can be phrased in this quantum setting, where we require them to be satisfied with probability $1$. For the precise way of how the questions and conditions are constructed in the quantum case, we refer the reader to Definitions 5.2 and 5.4 in \cite{brannan2022quantum}. For our purposes it suffices to state that the existence of a winning strategy amounts to the following statement, which may be seen as a definition of colouring for quantum graphs.

\begin{definition}[cf. Definition 2.15 in \cite{ganesan_spectral_2023} and Theorem 4.7 in \cite{brannan2022quantum}]\label{def:bghcol}
  Let $G$ be a quantum graph with corresponding operator space $S$. Then $G$ has a \emph{BGH $k$-colouring} if there exist projections $p_1, \dots, p_k \in V(G)$ with $\sum_s p_s = \identity$ such that
  \begin{equation*}
    p_s S p_s = 0  
  \end{equation*}
  for all $s \in [k]$.
\end{definition}

The projectors here correspond precisely to the projective measurement that guarantees a winning probability of $1$. Using the phrasing of \cref{def:bghcol}, it is not hard to show that it is equivalent to our \cref{def:colouring} in terms of projective subsets.

\begin{proposition}\label{prop:colouringbghequiv}
  A quantum graph $G$ has a $k$-colouring if and only if it has a BGH $k$-colouring.
  \begin{proof}
    Recall that projections $p_1, \dots, p_k \in V(G)$ correspond precisely to projective subsets $P_1, \dots, P_k \psubseteq V(G)$ such that $\pproj{P_s} = p_s$. Vectorising the elements $x \in S \subseteq M_n$, we can rephrase the condition $p_s S p_s = 0$ as
    \begin{equation*}
      \input{diagrams/bgh-colouring-cond.tikz}\enspace.
    \end{equation*}
    The diagram merits some clarification. A priori, the input $x$ of $\pproj{E(G)}$ is allowed to range over all of $M_n$, not just over the space $T$ of right $V(G)'$-module homomorphisms. Such a choice would violate the topology of the bicategorical diagrams, as the shading would have to be terminated without the presence of a separating $1$-morphism. However, note that since $\img \pproj{E(G)} \subseteq T$, we have $\ker \pproj{E(G)} \supseteq T^\bot$. Since every $x \in M_n$ can be written as $x = x' + x''$ with $x' \in T$ and $x'' \in T^\bot \subseteq \ker \pproj{E(G)}$, we may restrict to $T$. For the same reason, we may omit the $x$ altogether, which yields
    \begin{equation*}
      \input{diagrams/bgh-colouring-cond-rewr-1.tikz} \qquad\IFF\qquad \input{diagrams/bgh-colouring-cond-rewr-2.tikz}.
    \end{equation*}
    The latter is precisely the condition that $E(G)$ and $P_s \times P_s$ are disjoint, as desired. By \cref{lem:partchar}, the sum condition corresponds precisely to $P_1, \dots, P_k$ forming a partition of $V(G)$, which completes the proof.
  \end{proof}
\end{proposition}

As mentioned above, \cref{def:bghcol} is just a rephrasing of the existence of a winning strategy in the colouring game, so we may explicitly state the following.

\begin{corollary}\label{cor:colouringgame}
  A quantum graph $G$ has a $k$-colouring if and only if there exists a winning strategy in the colouring game.
\end{corollary}

We can also characterise the colouring definition in terms of homomorphisms. In fact, the colouring game of \cite{brannan2022quantum} was introduced as a special case of the \emph{(quantum-to-classical) graph homomorphism game}, which is a quantum generalisation of the classical graph homomorphism game. Classically, instead of answering with a colour, Alice and Bob have to answer with a vertex of a target graph $H$. If the players received adjacent vertices, they have to answer with vertices that are adjacent in $H$. This way, one recovers the colouring game by setting $H = K_k$. Analogously to the colouring game, there exists a local winning strategy for a classical graph $G$ in the classical $H$-homomorphism game if and only there exists a homomorphism from $G$ to $H$. The quantum-to-classical game extends this to quantum graphs $G$: There is a winning strategy in the quantum-to-classical $H$-homomorphism game for a quantum graph $G$ if and only if there exists a homomorphism from $G$ to $H$.

However, there is an important caveat. Namely, the homomorphisms that correspond to the winning strategies of the homomorphism game are distinct from those introduced in \cref{def:qghom}. On a conceptual level, one has to distinguish between \emph{pure} homomorphisms, due to Musto, Reutter, and Verdon \cite{musto_compositional_2018}, which are classical functions in the sense of \cref{def:cfunc}, and \emph{mixed} homomorphisms, due to Stahlke \cite{stahle2016quantum}, which are quantum channels, or more precisely completely positive counital maps between quantum sets. One can show that pure homomorphisms are special cases of the more general mixed homomorphisms, but not vice versa, see also \cite[Section V.B]{musto_compositional_2018}. The quantum-to-classical homomorphism game characterises precisely the mixed homomorphisms.

These results immediately imply that a quantum graph $G$ is $k$-colourable if and only if there exists a mixed homomorphism from $G$ to $K_k$.  It turns out that the same result also holds for the pure definition of homomorphisms given in \cref{def:qghom}. This seems to be known to some researchers in the field, but we were unable to find a reference. We prove it here using the diagrammatic calculus. 

\begin{proposition}
  A quantum graph $G$ has a $k$-colouring if and only if there exists a homomorphism $G \to K_k$.
  \begin{proof}
    The edge relation of $K_k$ is given by the following projector $\C^k \otimes \C^k \to \C^k \otimes \C^k$.
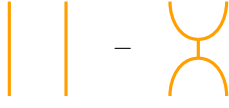
    \begin{equation*}
      \input{diagrams/edgerel-kn.tikz}
    \end{equation*}
    A homomorphism $\varphi\colon G \to K_k$ is thus a quantum function satisfying 
    \begin{equation*}
      \input{diagrams/kn-hom-cond.tikz}
    \end{equation*}
    This is satisfied if and only if
    \begin{equation*}
      \input{diagrams/kn-hom-cond-equiv.tikz}
    \end{equation*}
    and we can rephrase the condition in terms of $\varphi\adjoint \eqqcolon \psi$ by taking the adjoint of both sides.
    \begin{equation*}
      \input{diagrams/kn-hom-cond-equiv-adj.tikz}
    \end{equation*}
    Finally, we pass to the double-wire calculus in $\TFdHilb$, for which the homomorphism condition becomes 
    \begin{equation}\label{eq:qghomknequiv}
      \input{diagrams/kn-hom-cond-equiv-doublewire.tikz}
    \end{equation}

    Note that since $\varphi$ commutes with comultiplication and preserves the counit, $\psi\colon \C^k \to V(G)$ commutes with multiplication and preserves the unit. Moreover, $\psi$ is real, since the condition is invariant under taking adjoints. Since the standard basis elements are projections in $\C^k$, this implies that their images under $\psi$ are projections in $V(G)$. Moreover, since $\psi$ is unital they sum to the identity. Conversely, we can define a map satisfying the three abovementioned properties from any collection of projectors summing to the identity by setting them as the images of the standard basis elements. We claim that the projectors given by the images of the standard basis elements satisfy the colouring condition if and only if $\psi\adjoint = \varphi$ is a graph homomorphism $G \to K_k$. We have seen that $\varphi$ is a graph homomorphism if and only if \cref{eq:qghomknequiv} holds. Define 
    \begin{equation*}
      \input{diagrams/kn-hom-cond-pdef.tikz}
    \end{equation*}
    for all $s \in [k]$. Since $\psi(\ket{s})$ is a projection in $V(G) \subseteq B(H)$, $\pproj{P_s}\colon H \to H$ is a projector that is also a right-module homomorphism over $V(G)'$. Bending down the outer wires of \cref{eq:qghomknequiv} and plugging in the standard basis yields for all $s, t \in [k]$
    \begin{equation*}
      \input{diagrams/kn-hom-cond-equiv-final.tikz},
    \end{equation*}
    where the second equality follows from the realness of $s$ and $\psi$, and the fact that taking conjugates in the diagrammatic calculus amounts to mirroring across a vertical axis. The last equality then follows from the definition of $\pproj{P_s}$ and is precisely the condition that $E(G)$ and $P_s \times P_s$ be disjoint. Since the $\pproj{P_s}$ sum to the identity, we conclude by \cref{lem:partchar} that $P_1, \dots, P_k$ moreover form a partition of $V(G)$, recovering \cref{def:colouring}.
  \end{proof}
\end{proposition}

\subsection{Independent Sets and Cliques}
Finally, let us turn our attention to independent sets and cliques. In the spirit of the definitions we have stated so far, an independent set should be given by a projective subset $X \psubseteq V(G)$ such that the set $X \times X$ of all possible internal edges in $X$ is disjoint with $E(G)$. 

\begin{definition}\label{def:indepsetsloopless}
  An \emph{independent set} of size $k$ of a quantum graph $G$ is given by a projective subset $X \psubseteq V(G)$ with $\abs{X} = k$ such that $X \times X$ and $E(G)$ are disjoint. The size of the largest independent set of $G$ is called the \emph{independence number} and is denoted by $\alpha(G)$.
\end{definition}

Similarly, we can state a definition of cliques. While an independent set is given by a subset $X$ of vertices such that none of their internal edges are contained in the graph, a clique is a subset where all internal edges are contained in the graph. In order to correctly generalise the classical case, we need to take some care when considering loops. The complete graph $X \times X$ on $X$ contains self-loops, while $G$ does not. Instead of comparing $X \times X$ to $E(G)$, we have to compare it to $E(\loops{G})$, the quantum graph $G$ with all loops added.

\begin{definition}\label{def:cliquesloopless}
  A \emph{clique} of size $k$ of a quantum graph $G$ is given by a projective subset $X \psubseteq V(G)$ such that $X \times X \psubseteq E(G^\circ)$. The size of the largest clique of $G$ is called the \emph{clique number} and is denoted by $\omega(G)$.
\end{definition}

We first note that the definition of independent sets is quite close to that of colouring. In fact, it is easy to see that a $k$-colouring is precisely a partition of the vertex set into $k$ independent sets. We thus immediately recover a classically well-known bound on the chromatic number.\footnote{Note that the $\abs{V(G)}$ here refers to the (subset) cardinality of $V(G)$, not its dimension.}

\begin{proposition}\label{prop:chromnumineq}
  Let $G$ be a quantum graph. Then $\chi(G) \geq \abs{V(G)} / \alpha(G)$.
\end{proposition}

We also see that, just as in the classical case, every independent set of a quantum graph $G$ is a clique of $\compl{G}$ and vice versa.

\begin{proposition}
  Let $G$ be a quantum graph. A projective subset $X \psubseteq V(G)$ is an independent set of $G$ if and only if it is a clique of $\compl{G}$.
  \begin{proof}
    Per definition, $X$ is an independent set of $G$ if and only if $\img \overline{\pproj{X}} \otimes \pproj{X} \subseteq (\img \pproj{E(G)})^\bot$. It is a clique of $\compl{G}$ if and only if
    \begin{align*}
      \img \overline{\pproj{X}} \otimes \pproj{X} &\subseteq \img \pproj{E(\compl{G})} \cup \img \pproj{\Delta}\\
                                                  &= ((\img \pproj{E(G)})^\bot \cap (\img \pproj{\Delta})^\bot) \cup \img \pproj{\Delta}\\ 
                                                  &= (\img \pproj{E(G)})^\bot,
    \end{align*} 
    where the last equality holds because we have $\img \pproj{\Delta} \subseteq (\img \pproj{E(G)})^\bot$ since $G$ has no loops. 
  \end{proof}
\end{proposition}

The discussion of independent sets in the quantum graph literature is much more muddy than for colourings and connected components. At least five different, non-equivalent definitions of independent sets have been considered, cf. \cite[Proposition 2]{duan_zeroerror_2013}. These definitions are of a very operational flavor, as they correspond to properties of the quantum channels that have the quantum graph as their confusability graph. We recall here the two most fundamental definitions, which correspond respectively to the one-shot classical and quantum zero-error capacities of quantum channels.

\begin{definition}[Equation (3) in \cite{duan_zeroerror_2013}]\label{def:dswindepset}
  Let $G$ be a quantum graph with loops at every vertex and corresponding operator space $S \subseteq B(H)$. A \emph{DSW independent set} of size $k$ of $G$ is given by a collection of orthonormal vectors $\psi_1, \dots, \psi_k \in H$ such that $\braket{\psi_i | S | \psi_j} = 0$ for all $i \neq j$. The size of the largest such independent set of $G$ is denoted by $\alpha_c(G)$.
\end{definition}

The motivation for this definition is information theoretic. It is a basic result in quantum information theory that two pure quantum states represented by norm-$1$ vectors $\psi_i, \psi_j \in \C^n$ are perfectly distinguishable if and only if they are orthogonal. One can now show that given a quantum channel $\Phi \colon M_n \to M_m$ with confusability graph $S$, the states stay perfectly distinguishable after applying $\Phi$, that is $\Phi(\ket{\psi_i}\bra{\psi_i}) \bot \Phi(\ket{\psi_j}\bra{\psi_j})$, if and only if $\braket{\psi_i | S | \psi_j} = 0$. A DSW independent set thus corresponds to an encoding of $k$ classical messages that can be transmitted without error using $\Phi$. This is sometimes called the classical one-shot zero-error capacity of the channel. A definition due to Weaver provides the quantum analogue of this notion.

\begin{definition}[cf. Definition 5.2 in \cite{weaver2017quantum}, Definition 2.16 in \cite{spitzer_quantum_2026}]
  Let $G$ be a quantum graph potentially containing loops with corresponding operator space $S$. An \emph{independent set of size $k$ in the sense of Weaver} is given by a projector $p \in V(G)$ with $\rk p = k$ such that $pSp \subseteq p V(G)' p$. The size of the largest such independent set of $G$ is denoted by $\alpha_q(G)$.
\end{definition}

Recall that the projector onto $S$ is given by $\pproj{E(G)}$. If $G$ has loops at every vertex, as in the case of confusability graphs, we have $\Delta \psubseteq E(G)$ and thus $\img \pproj{\Delta} \subseteq S$. For $V(G) = M_n$, we have $\img \pproj{\Delta} = V(G)' = \C \identity$, so $S$ contains the identity and thus the independent set condition becomes $pSp = \C p$. This is precisely the Knill-Laflamme error correction condition, which captures the existence of an encoding- and a decoding map that allows to transmit an entire state subspace of dimension $\rk p$ without error. This is called the quantum one-shot zero-error capacity. It is known, and in fact not hard to show, that the quantum capacity is at most the classical capacity. This implies that for quantum graphs $G$ on $M_n$, we have $\alpha_q(G) \geq \alpha_c(G)$. For a more detailed explanation of the quantum information-theoretic background, we refer the reader to Section III in \cite{duan_zeroerror_2013}. 

Weaver proceeds to define cliques in a complementary way. If $S \supseteq V(G)' = \img \pproj{\Delta}$, then $pV(G)'p$ is the smallest space that can contain $pSp$. Consequently, to define cliques one may consider the largest space that can be contained in $pSp$, which is given by $p M_n p$.

\begin{definition}[cf. Definition 5.2 in \cite{weaver2017quantum}, Definition 2.17 in \cite{spitzer_quantum_2026}]\label{def:cliques}
  Let $G$ be a quantum graph potentially containing loops with corresponding operator space $S$. A \emph{clique of size $k$ in the sense of Weaver} is given by a projector $p \in V(G)$ with $\rk p = k$ such that $p(S + \img \pproj{\Delta})p = p M_n p$. The size of the largest such clique of $G$ is denoted by $\omega_q(G)$.
\end{definition}

\begin{remark}
Technically, Weaver's original definitions of independent sets and cliques in \cite{weaver2017quantum} apply only to graphs with loops at every vertex. We modified both definitions slightly to accomodate quantum graphs without loops, as was done in \cite{spitzer_quantum_2026}. The resulting definitions are equivalent when restricted to the former case.
\end{remark}

 Let us now compare our definitions of independent sets and cliques to those of Weaver, since his definitions have enjoyed the most attention among independent set and clique definitions \cite{weaver2017quantum,weaver2019quantum,Allen12122025,spitzer_quantum_2026}. We first show that our independent set definition is strictly stronger.

\begin{proposition}\label{prop:indepsetweaversep}
  Let $G$ be a quantum graph. Then every independent set of $G$ is also an independent set in the sense of Weaver. In particular, we have $\alpha(G) \leq \alpha_q(G)$. Moreover, there exists a quantum graph $G$ such that $\alpha(G) < \alpha_q(G)$
  \begin{proof}
    Note that using the same proof as in \cref{prop:colouringbghequiv} one can show that letting $p \coloneqq \pproj{X}$, \cref{def:indepsetsloopless} is equivalent to $pSp = 0$, and we have $0 \subseteq p V(G)' p$. This proves the first part.

    To show the second part, we construct a quantum graph $G$ such that $\alpha_q(G) \geq n - 1$ and $\alpha(G) \leq 1$. To that end, consider the operator space $S = \C x$, where $x = I - n \ket{n}\bra{n}$. We view $S$ as a quantum graph $G$ on $M_n$. Since $S$ contains only traceless matrices, we have that $S \bot \C \identity = \img \pproj{\Delta}$, so $G$ is loopless. Moreover, $x$ is self-adjoint, so $S$ is closed under adjoints and hence $G$ is undirected. Now consider the projector
    \begin{equation*}
      p = \identity - \ket{n}\bra{n}.
    \end{equation*}
    Clearly, we have $pxp = p$, so $p$ witnesses an independent set in the sense of Weaver of size $\rk p = n - 1$. On the other hand, there cannot be a projector $q$ of rank $\geq 2$ such that $qSq = 0$. Indeed, any vector $v \in \img q$ would have to satisfy 
    \begin{equation*}
      0 = \braket{v | x | v} = \norm{v}^2 - n \abs{v_n}^2,
    \end{equation*}
    or equivalently $\norm{v}^2 = n \abs{v_n}^2$. If $\rk q \geq 2$, then there exist linearly independent vectors $u, v \in \img q$, so $w \coloneqq v - v_n u / u_n \in \img q$ satisfies $w \neq 0$ and $w_n = 0$. But then $0 \neq \norm{w}^2 = n \abs{w_n}^2 = 0$, contradiction.
  \end{proof}
\end{proposition}

The same holds for our clique definition.

\begin{proposition}\label{prop:cliqueweaversep}
  Let $G$ be a quantum graph. Then every clique of $G$ is also a clique in the sense of Weaver. In particular, we have $\omega(G) \leq \omega_q(G)$. Moreover, there exists a quantum graph $G$ such that $\omega(G) < \omega_q(G)$.
  \begin{proof}
    Let $S$ be the operator space associated to the quantum graph $G$. Let $X \psubseteq V(G)$ and $p \coloneqq \pproj{X}$. Observe that $\img \pproj{X \times X} = \img \overline{p} \otimes p = p M_n p$. $X$ is thus a clique in the sense of \cref{def:cliquesloopless} if and only if
    \begin{equation}\label{eq:cliquedef}
      p M_n p \subseteq S + \img \pproj{\Delta}.
      \tag{P}
    \end{equation}
    On the other hand, a clique in the sense of Weaver is given by a $p$ such that $p (S + \img \pproj{\Delta}) p = p M_n p$.
    Since the left-hand side is trivially contained in the right-hand side, the definition is equivalent to 
    \begin{equation}\label{eq:cliquedefweaver}
      p M_n p \subseteq p (S + \img \pproj{\Delta}) p.
     \tag{W}
    \end{equation}
    Clearly \ref{eq:cliquedef} implies \ref{eq:cliquedefweaver}, which shows that $\omega(G) \leq \omega_q(G)$. The fact that there exists a quantum graph such that $\omega(G) < \omega_q(G)$ will follow from \cref{prop:xaimpr}.
  \end{proof}
\end{proposition}

Given their strong operational interpretation, Weaver's definitions are quite useful, so it prima facie seems disappointing to recover a different notion. However, we argue that our definition has certain advantages from a graph-theoretic perspective.

The first drawback of Weaver's definition compared to ours is that it loses the familiar property that independent sets in $G$ are cliques in $\compl{G}$. Indeed, consider the operator space $S_D$ of traceless diagonal matrices, viewed as a quantum graph $G_D$ on $M_n$. Weaver proved \cite[Proposition 2.2]{weaver2017quantum} that if $n \geq k^2 + k - 1$, then $\omega_q(G_D) \geq k$. On the other hand, the complement $\compl{G_D}$ is given by the operator space $S_D' = \{\ket{i}\bra{j} \mid i \neq j \in [n]\}$, and we can show that $\alpha_q(\compl{G_D}) \leq 1$ for all $n$. Indeed, suppose there was a projector $p \in M_n$ with $\rk p \geq 2$ such that $pS_D'p \subseteq \C p$. Then there exist two orthonormal vectors $u, v \in \img p$ such that for all $x \in S_D'$ we have
\begin{equation*}
  \braket{v | x | u} = \braket{v | p x p | u} = \lambda_x \braket{v | u} = 0.
\end{equation*}
This in turn is equivalent to $\ket{u}\bra{v} \bot S_D'$, so $\ket{u}\bra{v}$ must be diagonal. But this is impossible for non-zero orthogonal $u, v$. It follows that $\alpha_q(G_D) < \omega_q(\compl{G_D})$ for $n \geq 5$, so there must be cliques in $G_D$ that are not independent sets in $\compl{G}_D$.

Another drawback consists in the interaction of independent sets with other graph-theoretic definitions, in particular colouring. We have seen that a colouring is a covering of the graph with independent sets precisely in our sense. Of course one has to remark that it is straightforward to state an alternative definition of colouring which amounts to a covering of the graph with independent sets in Weaver's sense. Such a definition, however, would not enjoy a characterisation in terms of the colouring game proposed by Brannan, Ganesan, and Harris, cf. \cref{cor:colouringgame}.

As a final remark, we also mention that compared to Weaver's cliques, our definition moreover somewhat cleans up a minor result by Spitzer and Nechita. In \cite{spitzer_quantum_2026}, the authors introduce explicit families of quantum graphs on $M_n$ and study their graph theoretic properties. One of these families is given by the quantum graphs $X_{A, \ccdot}$, where $A$ is the adjacency matrix of a classical graph. Their associated operator spaces are precisely given by those associated to the underlying classical graphs (cf. \cref{eq:clgraphopspace}), we simply view them as bimodules over $M_n'$ instead of $(\C^n)'$, as described in the remark in \cref{ssec:conncomp}. It is shown that quantum graphs of this form behave very similarly to their associated classical graph: They have the same number of connected components, the same chromatic number, and the same independence number (in the sense of Weaver), cf. \cite[Table 1]{spitzer_quantum_2026}. Suprisingly, the correspondence fails for the clique number. It is shown that $\omega(A) - 1 \leq \omega_q(X_{A, \ccdot}) \leq n - 1$, with both bounds being tight. It is unclear how to intepret this deviation from the classical clique number. For our definitions we instead obtain $\omega(X_{A, \ccdot}) \in \{0, 1\}$ while preserving the equality $\alpha(X_{A, \ccdot}) = \alpha(A)$. The former may be interpreted as the essentially classical $X_{A, \ccdot}$ not being ``noncommutatively dense'' enough to admit large cliques, cf. also \cite[Section 4.2.5]{spitzer_quantum_2026}.

\begin{proposition}\label{prop:xaimpr}
  For all classical graphs $A$ it holds that
  \begin{equation*}
    \alpha(X_{A, \ccdot}) = \alpha(A) \qquad\text{and}\qquad \omega(X_{A, \ccdot}) = \begin{cases}
      1 & \text{if } A = K_n\\
      0 & \text{otherwise.}
    \end{cases}
  \end{equation*}
  \begin{proof}
    First note that since the operator spaces of $X_{A, \ccdot}$ and $A$ coincide, and the class of projective subsets of $M_n$ subsumes that of projective subsets of $\C^n$ we immediately find that $\alpha(X_{A, \ccdot}) \geq \alpha(A)$. Note that the same does not hold for the clique number $\omega$. This is because the definition of cliques includes the diagonal relation $\Delta$, which on $M_n$ is given by $\img \pproj{\Delta} = \C \identity$, while on $\C^n$ it is given by $\img \pproj{\Delta} = D_n$. At the same time, we have $\alpha(X_{A, \ccdot}) \leq \alpha_q(X_{A, \ccdot}) = \alpha(A)$ by \cref{prop:indepsetweaversep} and \cite[Proposition 4.28]{spitzer_quantum_2026}, which shows that $\alpha(X_{A, \ccdot}) = \alpha(A)$.

    For the clique number, note that the projector into the operator space associated to $X_{A, \ccdot}$ is given by
    \begin{equation*}
      \input{diagrams/xa-edgerel.tikz}\enspace.
    \end{equation*}
    It follows that a projective subset $X \psubseteq V(G)$ is a clique of $X_{A, \ccdot}$ if and only if 
    \begin{equation*}
      \input{diagrams/xa-clique-cond.tikz}\enspace
    \end{equation*}
    Plugging in standard basis vectors at the bottom gives the following equivalent condition.
    \begin{equation}\label{eq:xacliquecondrewr}
      \input{diagrams/xa-clique-cond-basis.tikz}
    \end{equation}
    For $i = j$, the $A$ term vanishes since $A$ has no loops. In that case, by isomorphism of diagrams, the following condition must be satisfied.
    \begin{equation}\label{eq:xacliqueconddiag}
      \input{diagrams/xa-clique-cond-diag.tikz}
    \end{equation}
    This shows that $\pproj{X}$ can have rank at most $1$. Conversely, a clique of size $1$ is achievable if and only if $A$ is the complete graph. Indeed, note that if $A$ is not the complete graph, then there exist $i \neq j$ such that $A_{ij} = A_{ji} = 0$. In that case, the only way to satisfy \cref{eq:xacliquecondrewr} is if either $\pproj{X}\ket{i} = 0$ or $\pproj{X}\ket{j} = 0$. But then \cref{eq:xacliqueconddiag} implies that $\pproj{X} = 0$. On the other hand, if $A$ is the complete graph then the condition for $i \neq j$ is vacuous, and \cref{eq:xacliqueconddiag} is satisfied if $\pproj{X}$ is the projection onto the all-ones vector.
  \end{proof}
\end{proposition}

We have seen that when generalising classical independent sets in the spirit of how we generalised connected components and colourings, we recover a strictly stronger notion than that of Weaver, but has certain desirable properties.
Nevertheless, there is also a way of interpreting Weaver's independent sets in terms of projective subsets. The idea is to grant loops of quantum graphs a more central role. So far, all our definitions only apply to quantum graphs without loops. This is not a meaningful restriction, since we may always remove them as shown in \cref{cor:loopaddremove}. For instance, the independent sets of a quantum graph $G$ containing loops are simply defined as the independent sets of its loopless counterpart $\looprm{G}$. The alternative approach is to accept that $G$ has loops and instead of checking whether $X \times X$ and $E(G)$ are disjoint---which classically is never the case if $G$ contains loops---to check whether $(X \times X) \setminus \Delta$ and $E(G)$ are disjoint.

\begin{definition}\label{def:weakindepset}
  Let $G$ be a quantum graph potentially containing loops. A \emph{weak independent set} is given by a projective subset $X \psubseteq V(G)$ such that $(X \times X) \setminus \Delta$ and $E(G)$ are disjoint.
\end{definition}

Observe that restricted to the classical case, this definition is completely equivalent to \cref{def:indepsetsloopless}. If a classical graph $G$ has loops, then it does not matter whether we remove the loops of $G$ and compare $X \times X$ to the resulting edge relation, or remove the loops of the complete reflexive graph $X \times X$ and compare it to the original edge relation of $G$. Quantumly, however, these two approaches are no longer equivalent, and it turns out that it is the latter approach that yields independent sets in the sense of Weaver. 

\begin{proposition}
  Let $G$ be a quantum graph potentially containing loops. Then $X$ is a weak independent set if and only if $\pproj{X}$ is an independent set in the sense of Weaver.
  \begin{proof}
    Written in terms of subspaces, the weak independent set condition reads
    \begin{equation}\label{eq:weakindepsetcond}
      \img \overline{\pproj{X}} \otimes \pproj{X} \cap (\img \pproj{\Delta})^\bot \;\bot\; S,
    \end{equation}
    where $S = \img \pproj{E(G)}$ is the operator space corresponding to $G$. For simplicity, denote $p = \pproj{X}$, $V = \img \overline{p} \otimes p$ and $D = \img \pproj{\Delta}$. Note that per definition we have $D = V(G)'$. Moreover, by vectorisation, the elements in $V$ are precisely of the form $pxp$ where $x \in M_n$.

    Now suppose that $v \in V \cap D^\bot$ and $s \in S$. Then we have $\lrangle{v, s} = \lrangle{p v p, s} = \lrangle{v, psp}$, so it follows that \cref{eq:weakindepsetcond} is equivalent to $V \cap D^\bot \;\bot\; pSp$. This in turn is equivalent to 
    \begin{equation*}
      pSp \subseteq (V \cap D^\bot)^\bot = V^\bot \cup D.
    \end{equation*}
    Since $pSp$ is obviously orthogonal to $V^\bot$, we must have $pSp \subseteq D$. At the same time, we have
    \begin{equation*}
      D = (p + p^\bot) D (p + p^\bot) = pDp + pDp^\bot + p^\bot D p + p^\bot D p^\bot.
    \end{equation*}
    Since $pSp$ is again obviously orthogonal to $p^\bot D p^\bot$, $p^\bot D p$, $p D p^\bot$, we must have
    \begin{equation*}
      pSp \subseteq pDp = pV(G)'p.
    \end{equation*}
    Since $p = \pproj{X}$, this is precisely the condition that $\pproj{X}$ is an independent set in the sense of Weaver.
  \end{proof}
\end{proposition}

One might hope that a similar result holds for cliques. In \cref{def:cliques}, we modify the loops of the quantum graph $G$ in order to obtain the correct inclusion conditions, just as we did in \cref{def:indepsetsloopless} of independent sets. The other option would be to modify the loops of the graph $X \times X$ instead, which would yield the condition
\begin{equation}\label{eq:altcliquecond}
  (X \times X) \setminus \Delta \psubseteq E(G).
\end{equation}
Unfortunately, it turns out that it is strictly stronger than Weaver's definition. Moreover, the condition is incomparable with our \cref{def:cliques} of cliques, so we cannot justify calling them ``weak cliques'' in our setting. We omit the proofs of these two statements for the sake of brevity.
On the positive side, the definition is nicely dual to \cref{def:weakindepset} of weak independent sets, and it is not hard to prove that a clique in this sense is a weak independent set of the complement graph and vice versa. The pairing of weak independent sets and cliques in the sense of \cref{eq:altcliquecond} might ultimately turn out to have more desirable properties than the pairing of \cref{def:indepsetsloopless} and \cref{def:cliques}. Both pairings behave nicely with respect to the graph complement, and while the latter is closely related to colourings, the former recovers Weaver's independent sets and hence inherits their operational interpretation. We leave this avenue of inquiry for future work.

\subsubsection{On operational interpretations}\label{sssec:indepopinterp}
Given the strong operational interpretations of DSW and Weaver independent sets, a natural question to ask is whether our definition of (strong) independent sets also admits such an interpretation. We do not fully answer this question, but we give some evidence towards a positive answer.   

The crucial observation is that there are multiple ways of associating a quantum graph to a quantum channel. The approach that is used to give an operational interpretation to DSW and Weaver independent sets considers confusability graphs. The other quantum graph we may associate to a quantum channel we call the \emph{mapping graph}. The idea stems from classical channels. Let $X$ be a set of code words. A classical channel $\mathcal{N}$ may be considered as taking a word $w \in X$ and returning any given word $w' \in X$ with (potentially zero) probability $p_{\mathcal N}(w' \mid w)$. The channel is then uniquely defined by the probability matrix $(p(x\mid y))_{x, y \in X}$. We may associate to such a channel a directed graph whose vertex set is $X$ and that contains a directed edge from $y$ to $x$ if and only if $p_{\mathcal N}(x \mid y) \neq 0$. Such a graph captures much more about the channel than the confusability graph. We can lift this concept to the quantum case by associating to a quantum channel $\Phi\colon M_n \to M_n$ with Kraus operators $F_1, \dots, F_k$ the operator space
\begin{equation*}
  S^\Phi = \lspan \{F_i \mid i \in [k] \},
\end{equation*}
which we call the \emph{quantum mapping graph} of $\Phi$.  

The idea to consider this graph is not new, cf. Section VII in \cite{duan_zeroerror_2013}, but it has received significantly less attention than its confusability counterpart. We show that we may use this formalism to give an interpretation of (strong) independent sets. The idea is based on work by Bei, Chen, Guan, Qiao, and Sun \cite{bei_independent_2021}. In their paper, they introduce characterisations of colourings and independent sets of classical graphs in terms of matrix spaces. This is obviously closely related to some of the definitions mentioned in the previous sections; the difference being that they mainly work over arbitrary fields instead of over the complex numbers. As a consequence, all their definitions use transposes instead of adjoints, which makes them unnatural in the quantum case. The authors seem to agree, and they address the complex case separately at the end of their paper. Concretely, they consider sets of Kraus operators of quantum channels and define the following.

\begin{definition}[Definition 13.1 in \cite{bei_independent_2021}]\label{def:isospace}
  Let $B = \{B_1, \dots, B_m\}$ be the set of Kraus operators of a quantum channel. An \emph{isotropic space} of $B$ is a subspace $U \subseteq \C^n$ such that for any $u, u' \in U$ and any $B_i \in B$, we have $\braket{u | B_i | u'} = 0$.
\end{definition}

They then associate to a classical graph $G$ the matrix set
\begin{equation*}
  B_G \coloneqq \left\{\frac{1}{\sqrt{\deg{i}}} \ket{i}\bra{j}, \frac{1}{\sqrt{\deg{j}}} \ket{j}\bra{i} \mid i \sim j \in [n] \right\},
\end{equation*}
and show that isotropic spaces of $B_G$ correspond to independent sets of $G$.

\begin{proposition}[Proposition 13.3 in \cite{bei_independent_2021}]\label{prop:isospaceclindepsetequiv}
  Let $G$ be a classical graph and let $B_G$ be as above. Then $G$ has an independent set of size $k$ if and only if $B_G$ has an isotropic space of dimension $k$.
\end{proposition}

The two definitions are readily linearised. Taking the span of the set $B$, \cref{def:isospace} is equivalent to the existence of a projection $p \in M_n$ such that $p (\lspan B) p = 0$. We have seen that this is equivalent to the graph with operator space $\lspan B$ having an independent set of size $\rk p$. At the same time, $\lspan B$ is precisely the operator space corresponding to the mapping graph of the quantum channel with Kraus operators $B_1, \dots, B_m$. Similarly, the span of the set $B_G$ is just equal to $S_G$, the operator space associated to a classical graph $G$, see \cref{eq:clgraphopspace}. \cref{prop:isospaceclindepsetequiv} for example implies that $\alpha(X_{A, \ccdot}) = \alpha(A)$, as we have seen in \cref{prop:xaimpr}. Our intuition about a potential operational interpretation of (strong) independent sets is now informed by the following result.

\begin{proposition}[Proposition 13.10 in \cite{bei_independent_2021}]
  Let $B = \{B_1, \dots, B_m\}$ be the set of Kraus operators of a quantum channel $\Phi$. Then $U$ is an isotropic space of $B$ if and only if $F^{max}_U(\Phi) = 0$.
\end{proposition}

Here $F^{max}_U(\Phi)$ is called the quantum gate maximum subspace-fidelity and is defined as $F^{max}_U(\Phi) = \max_{u \in U} \Tr \ket{u}\bra{u} \Phi(\ket{u}\bra{u})$. Intuitively, it measures how well $\Phi$ implements the identity channel on the most suitable quantum state in the subspace $U$. If the maximum subspace-fidelity is $0$, then every state in the subspace gets maximally perturbed. 

There are some obstacles towards fully formalising this idea. The main problem is that the definition of strong independent sets only applies to loopless quantum graphs. On the level of mapping graphs, this corresponds precisely to quantum channels whose Kraus operators are traceless. Such quantum channels are rare. One may show that these are the quantum channels that have zero entanglement fidelity with respect to the maximally mixed state, cf. \cite[Equation 43]{schumacher1996sending}. To define independent sets for quantum graphs with loops, we usually ignore the loops by considering the intersection of the operator space with the orthogonal complement of the diagonal relation. In the case of $M_n$, this amounts to passing to the subspace of traceless matrices. Unfortunately, this procedure is not physical: Consider a pure channel of a single unitary Kraus operator $U$ with non-zero trace. The operator system spanned by $U$ corresponds to a quantum graph with loops and removing the loops leaves the trivial subspace. The trivial subspace does not have a basis of Kraus operators and is consequently not the mapping graph of a quantum channel. This raises the question whether there is a physically meaningful way to associate a loopless mapping graph to an arbitrary quantum channel.

A second curiosity is that the proposed operational interpretation does not quite match the intuition behind the idea of the mapping graph. If the mapping graph roughly encodes which inputs may be mapped to which outputs, then an independent set intuitively should not only consist of states that get maximally perturbed from themselves, but also from every other state in the independent set, satisfying 
\begin{equation*}
  \max_{u,v \in U} \Tr \ket{v}\bra{v} \Phi(\ket{u}\bra{u}) = 0
\end{equation*}
instead. This raises the question to what extent a potential operational interpretation of independent sets can be reconciled with the classical intuition of the mapping graph. We leave these questions open for future work. 

\begin{question}
  Does the definition of (strong) independent sets in terms of projective subsets admit an operational interpretation in terms of quantum channels? 
\end{question}

Of course the same question can be asked for the other graph properties we have considered. Connected components of the mapping graph, for instance, should intuitively correspond to invariant subspaces of the quantum channel. To make this formal, one has to consider connected components of \emph{directed} quantum graphs, while we have only considered the undirected case. This is somewhat involved, since even classical directed graphs admit two equally valid definitions of connectedness: \emph{Strong} connectedness, where every vertex needs to be reachable from any other along directed edges, and \emph{weak} connectedness, which corresponds to undirected connectedness of the symmetric closure of the graph. In \cite{courtney_connectivity_2025}, the authors define strong connectedness, while invariant subspaces should intuitively correspond to weak connectedness.

\begin{question}
  Do the definitions of connected components and colourings admit operational interpretations in terms of quantum channels? 
\end{question}

\subsection{Vertex Covers}
To illustrate the usefulness of the projective subset formalism one final time, we use it to generalise the notion of \emph{vertex covers} to the quantum case. Classically, a vertex cover of a graph $G$ is a set $X \subseteq V(G)$ such that for every edge $(u, v) \in E(G)$ either $u \in X$ or $v \in X$. This definition admits the following more succinct phrasing, which allows us to immediately generalise it to the quantum case.

\begin{definition}
  A \emph{vertex cover} of a quantum graph $G$ is a projective subset $X \psubseteq V(G)$ such that $E(G) \psubseteq (X \times V(G)) \cup (V(G) \times X) \cup (X \times X)$.
\end{definition}

The reader may want to verify that one indeed recovers the notion of vertex covers if $G$ is a classical graph. Now it is easy to show that for classical graphs, the complement of a vertex cover is an independent set and vice versa. It turns out that this is still true in the quantum case. 

\begin{proposition}
  Let $G$ be a quantum graph and $X \psubseteq V(G)$. Then $X$ is a vertex cover if and only if $\compl{X}$ is an independent set.
  \begin{proof}
    We first show that the projection corresponding to $P \coloneqq (X \times V(G)) \cup (V(G) \times X) \cup (X \times X)$ is given by $\pproj{P'} \coloneqq \identity \otimes \pproj{X} + \overline{\pproj{X}} \otimes \identity - \overline{\pproj{X}} \otimes \pproj{X}$: A routine calculation shows that $\pproj{P'}$ is indeed a projection. Further, if $V(G) = \bigoplus_i M_{n_i}$, then $\pproj{X}$ is an endomorphism of $\bigoplus_i \C^{n_i} \cong \C^N$, so we may assume without loss of generality that $\pproj{X}\colon \C^N \to \C^N$ for some suitable $N \in \N$.

    By definition, we have
    \begin{align*}
      \img \pproj{P} = \lspan \big(&\{\overline{v} \otimes \psi \mid v \in \img \pproj{X}, \psi \in \C^N\}\\
        \cup~ &\{\psi \otimes v \mid v \in \img \pproj{X}, \psi \in \C^N\}\\
      \cup~ &\{\overline{v} \otimes v \mid v \in \img \pproj{X}\}\big).
    \end{align*}
    It is easy to see that the image of $\pproj{P'}$ is contained in this space. Conversely, let $v$ be an arbitrary element in $\img \pproj{P}$, that is
    \begin{equation*}
      v = \sum_i \alpha_i \overline{v_i^1} \otimes w_i^1 + \sum_j \beta_j w_j^2 \otimes v_i^2 + \sum_k \gamma_k \overline{v_k^3} \otimes v_k^4,
    \end{equation*}
    where the annotated $v$ vectors are in $\img \pproj{X}$ and the annotated $w$ are arbitrary in $\C^N$. It suffices to show that $\pproj{P'}v = v$, and by linearity it suffices to show that this holds for the individual summands. We have
    \begin{equation*}
    \pproj{P'}\left(\overline{v_i^1} \otimes w_i^1\right) = \overline{v_i^1} \otimes \pproj{X} w_i^1 + \overline{v_i^1} \otimes w_i^1 - \overline{v_i^1} \otimes \pproj{X} w_i^1 = \overline{v_i^1} \otimes w_i^1
    \end{equation*}
     and
    \begin{equation*}
    \pproj{P'}\left(\overline{v_k^3} \otimes v_k^4\right) = \overline{v_k^3} \otimes v_k^4 + \overline{v_k^3} \otimes v_k^4 - \overline{v_k^3} \otimes v_k^4 = \overline{v_k^3} \otimes v_k^4,
    \end{equation*}
    the remaining case follows by symmetry. It follows that $\pproj{P'} = \pproj{P}$ as desired. The statement now follows from the following chain of equivalences.
    \begin{equation*}
      E(G) \psubseteq P
    \end{equation*}
    \begin{equation*}
      \pproj{E(G)} \circ \pproj{P} = \pproj{E(G)}
    \end{equation*}
    \begin{equation*}
      \pproj{E(G)} \circ (\identity \otimes \identity - \overline{\pproj{X}} \otimes \identity - \identity \otimes \pproj{X} + \overline{\pproj{X}} \otimes \pproj{X}) = 0\\
    \end{equation*}
    \begin{equation*}
      \pproj{E(G)} \circ ((I - \overline{\pproj{X}}) \otimes (I - \pproj{X})) = 0
    \end{equation*}
    \begin{equation*}
      \pproj{E(G)} \circ (\overline{\pproj{\compl{X}}} \otimes \pproj{\compl{X}}) = 0
    \end{equation*}
    The final statement is equivalent to $E(G)$ and $\compl{X} \times \compl{X}$ being disjoint, which is the definition of $\compl{X}$ being an independent set of $G$.
  \end{proof}
\end{proposition}

\section{Conclusion}\label{sec:conclusion}
In this paper, we have developed the theory of projective subsets, which seek to refine the usual notion of quantum subsets. We found that they behave in many ways like classical subsets: They admit the usual set-theoretic operations, satisfy corresponding laws (with the important exception of distributivity), and interact well with classical functions between quantum sets. We have then phrased the most important definitions in quantum graph theory using projective subsets. The approach to this turned out to be 
delightfully
straightforward: It suffices to phrase the classical definitions in terms of vertex sets and replace the word ``subset'' by ``projective subset''. We found that this way, despite their different origins, the definitions of connectivity, connected components, colourings and independent sets can all be unified in the same framework. We moreover extended this collection by a definition of vertex covers. 

At the same time, we found that the quantum generalisations arising from this procedure do depend on the initial choice of classical definition. The salient example of this is the definition of independent sets. Two classically equivalent definitions---differing only in their handling of loops---give rise to two different notions of independent sets of quantum graphs. One choice results in a notion that behaves similarly to the classical case, satisfying familiar duality results with respect to vertex covers and graph complements, and interacting well with other quantum graph parameters. The other choice recovers Weaver's operationally motivated definition, which is closely related to zero-error capacities of quantum channels. It is particularly surprising that a purely graph-theoretically motivated definition recovers a fundamental property of quantum channels.

Just as for independent sets, we also obtained two distinct notions of cliques depending on how their definition is phrased in terms of projective subsets. Contrary to independent sets, however, neither definition recovers Weaver's cliques, both being strictly stronger. Nevertheless, we found that the four definitions of cliques and independent sets phrased in terms of projective subsets come in pairs---each type of clique respectively coincides with one type of independent set in the complement graph.

Mathematically, we have only considered the finite-dimensional case. This is because finite quantum sets are sufficient for almost all applications, and the graphical calculus employed is only valid for pivotal categories like $\FdHilb$. Its infinite-dimensional counterpart $\Hilb$ fails to be pivotal, since infinite-dimensional Hilbert spaces are generally not naturally isomorphic to their double duals. However, we expect almost all of our observations to straightforwardly generalise to the infinite-dimensional case. In particular, the characterisation given in \cref{def:projsubsetweaver} is an adaptation of Weaver's binary quantum relations for infinite quantum sets and immediately generalises to the infinite-dimensional case by replacing the \Cstar-algebra by a von Neumann algebra.

A natural next step is to study the newly introduced definitions of independent sets, cliques, and vertex covers. An obvious open question is to what extent some of the nice results that have been shown of Weaver's independent sets and cliques hold for the newly introduced pairs of cliques and independent sets. In particular, it has been shown \cite{weaver2017quantum} that Weaver's independent sets and cliques satisfy an analogue of Ramsey's theorem. Classically, the theorem states that for every $k \in \N$, there exists an $N \in \N$ such that every graph of size at least $N$ either has an independent set or a clique of size $k$. It would be interesting to verify that the same holds for the other choices of definitions, and compare the resulting bounds on $N$. 

It would also be interesting look at quantum properties. In this work, we have only considered classical graph properties as they would be introduced in a first course on graph theory. In recent years, there has been an increasing interest in \emph{quantum} properties of classical graphs. As mentioned in \cref{rem:nonlocalstrats}, these properties are usually defined by first characterising existing definitions in terms of non-local games, and then relaxing the definition by allowing quantum strategies. Two of the most prominent definitions are quantum colouring \cite{cameron2007quantum} and quantum isomorphism \cite{atserias2019quantum}. Of course, it would be desirable to extend these definitions to quantum graphs. This has been done for quantum colourings in \cite{brannan2022quantum}. The story for quantum isomorphism is more complicated. Besides non-local games, it can be phrased in terms of quantum automorphism groups, a perspective which is sometimes more convenient. Both these definitions have been extended to quantum graphs \cite{brannan_quantum_2024} \cite{brannan2020bigalois,gromada_examples_2022}. Unfortunately, while the two perspectives are equivalent for classical graphs, this equivalence breaks for quantum graphs. A third approach was taken in \cite{musto_compositional_2018}. The observation is that one can quantise the classical functions of \cref{def:cfunc} by adding a Hilbert space wire through the boxes, turning the linear maps $X \to Y$ representing classical functions between quantums sets $X, Y$ into linear maps $\mathcal{H} \otimes X \to Y \otimes \mathcal{H}$. The resulting diagrammatic conditions then define a \emph{quantum function} \cite[Definition III.11]{musto_compositional_2018}. They proceed by defining quantum bijections as special cases of quantum functions, and finally quantum isomorphisms as quantum bijections that intertwine the adjacency operators \cite[Definition V.11]{musto_compositional_2018}.

This graphical approach is quite flexible. For example, it allows us to easily define quantum projective subsets as maps $\mmap{X}$ satisfying 
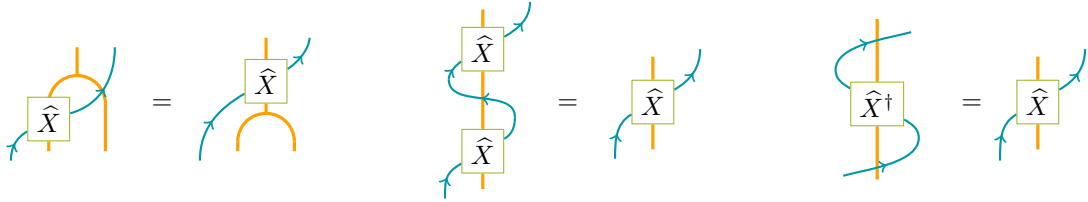
\begin{equation*}
  \input{diagrams/quantum-pset-mmap.tikz} \qquad\quad\qquad \input{diagrams/quantum-pset-mmap-idemp.tikz} \qquad\quad\qquad \input{diagrams/quantum-pset-mmap-adj.tikz}
\end{equation*}
It would be interesting to see whether one can recover the quantum colouring definition of \cite{brannan2022quantum} this way. Of course one would have to figure out how to define the notions of complements, intersections, and unions, since we did not define them in terms of diagrams. If this can be achieved, we would also obtain definitions of quantum independent sets, quantum cliques, quantum vertex covers, and quantum connected components.

\bigskip

\noindent\textbf{Acknowledgements.} The author is grateful for inspiring discussions with Adina Goldberg during the program ``Operator Algebras and Quantum Information'' at Mittag-Leffler Institute. The author moreover acknowledges support from the CNRS 80Prime grant ``QuantGraphe''.

\printbibliography
\stoptocwriting
  
\end{document}